\documentclass[12pt, reqno, twoside]{amsart}
\usepackage{amsfonts,amsthm,amsmath,amssymb,amscd}
\usepackage{mathabx,mathrsfs} 
\usepackage{mathtools}
\usepackage{graphicx}
\usepackage{tikz,float}
\usetikzlibrary{arrows.meta,calc,patterns,decorations.pathreplacing}
\usepackage{indentfirst}
\usepackage[T1]{fontenc}
\usepackage{microtype}
\usepackage{enumitem}
\usepackage{cases}
\usepackage[normalem]{ulem}

\usepackage[colorlinks=true, citecolor=blue, linkcolor=blue, urlcolor=blue]{hyperref}
\usepackage{cleveref}
\date{}
\allowdisplaybreaks[2]
\numberwithin{equation}{section}

\usepackage[a4paper, margin=1.25in, bottom=1.0in]{geometry}
\usepackage{fancyhdr}
\fancypagestyle{firstpage}{
	\fancyhf{}
}
\newcommand{\norm}[1]{\left\lVert #1 \right\rVert}

\newcommand{\abs}[1]{\left\lvert #1 \right\rvert}
\newcommand{\subeqref}[2]{$ \eqref{#1}_{#2} $}
\newcommand{\jump}[1]{\left\ldbrack#1\right\rdbrack} 
\newcommand{\p}{\partial}
\newcommand{\R}{\mathbb{R}}
\newcommand{\dt}{\frac{d}{dt}}

\newcommand{\ovl}{\overline}

\newcommand{\X}{\mathcal{X}}

\newcommand{\ull}{u_-}
\newcommand{\ur}{u_+}
\newcommand{\ubl}{\bar{u}_-}
\newcommand{\ubr}{\bar{u}_+}
\newcommand{\andd}{~\text{ and }~}
\def\ub{\bar{u}}
\def\xb{\bar{x}}

\def\vt{\tilde{v}}
\def\wt{\tilde{w}}
\def\usharp{u^{\sharp}}
\def\xiangle{\langle \xi-\xi_* \rangle}
\newcommand{\sgnk}{\mathrm{sgn}(k)}
\newtheorem{theorem}{Theorem}[section]

\newtheorem{lemma}[theorem]{Lemma}
\newtheorem{proposition}[theorem]{Proposition}
\newtheorem{remark}{Remark}[section]

\begin{document}

\title[Time-periodic boundary effects on conservation laws]{Nonlinear effects of incoming time-periodic boundary conditions on shock waves for scalar conservation laws}

	
	\author{Yuan~Yuan
	}
	\address[Yuan Yuan]{School of Mathematical Sciences, South China Normal University, Guangzhou 510631, China. }
	\email{\href{mailto:yyuan2102@m.scnu.edu.cn}{yyuan2102@m.scnu.edu.cn}}
	


\maketitle
\begin{abstract} 
	This paper is concerned with the large-time behavior of inviscid and viscous shocks for scalar conservation laws on the half-line $(-\infty,0)$ under an incoming \emph{time-periodic} boundary condition. 	
    This problem arises naturally from classical piston problems in fluid mechanics. 
    Even for arbitrarily small periodic oscillations of the boundary data, the effect of the boundary manifests itself through a persistent time-periodic component near the boundary, which behaves as a \emph{boundary wave}.  
    
    For inviscid shocks moving away from the boundary ($s<0$), this boundary wave repeatedly interacts with the shock over infinitely many periods. We show that the repeated interactions do not generate secular growth, but instead produce a finite \emph{cumulative} correction to the asymptotic shock location. For the inviscid stationary case $s=0$, the time-periodic boundary \emph{modulates} the background stationary shock into a genuinely time-periodic shock configuration, to which the solution converges.
    In the viscous case with $s<0$, the boundary wave couples with the viscous shock through a transitional region induced by diffusion, and the solution converges asymptotically to their superposition. Our results show that for scalar conservation laws, the persistent boundary wave modifies the shock without producing resonance.
	
\end{abstract}

\maketitle

\section{Introduction} 

In this paper, we study the asymptotic states of the solutions to the initial-boundary value problem (IBVP) for scalar conservation laws:
\begin{align}
	\p_t u(x,t)+\p_x f(u(x,t)) & = \mu \p_x^2 u, \quad x \in \R_-,~t>0, \label{eqn-0} \\
	u(0,t) &=u_b(t), \quad  t>0, \label{bc0}\\
	u(x,0) & = u_0(x), \quad x\in \R_-, \label{ic0}
\end{align}
where the flux $f(u)\in C^2(\R)$, $u_b(t)$ is a periodic function of period $T$, and $u_0(x)\rightarrow \ubl \text{~when~} x\rightarrow -\infty$. $\mu\geq 0$ is the viscosity constant, and in this paper it is set just as $0$ or $1$.

It is well-known that the large time behaviors of entropy solutions to the corresponding Cauchy problem (in the whole space $\R$) \eqref{eqn-0} and \eqref{ic0}, with the initial data approaching constant states $\ub_\pm$ as $x\rightarrow \pm \infty$, is governed by either shock waves, rarefaction waves, or  shock profiles, originating from the works of Oleinik, Lax, and others; see \cite{Ilin.O1960,Lax.1957} and the monographs \cite{Smoll.1983,Dafer.2016}.
When $f$ is strictly convex, when $\mu=0$ and the initial data are the Riemann data for constants $ \ubl > \ubr  $,
the unique entropy solution to the corresponding Cauchy problem \eqref{eqn-0},\eqref{ic0} is the well-known shock wave $ u^S(x,t) $: 
\begin{equation}\label{shock}
	u^S(x,t) := \begin{cases}
		\ubl, & \quad x<st,\\
		\ubr, & \quad x>st,
	\end{cases}
\end{equation} 
and $ s=s(\ubl,\ubr)$ is the shock speed satisfying the Rankine-Hugoniot condition 
\begin{equation}
	\label{RH} \tag{R-H}
	- \jump{\ub}s+ \jump{f(\ub)}=0,
\end{equation} 
and the Lax entropy condition
\begin{equation}
	\label{non-degenerate-shock}
	f'(\ubr)<s<f'(\ubl),
\end{equation}
where $\jump{u}:=\ur - \ull$, and we adhere to the convention that $f(u)_{\pm}:=f(u_\pm)$ for any function $f$.
When $\mu> 0$, the corresponding Cauchy problem has a unique viscous shock wave $\phi^S$ up to a shift with the shock speed $s$ mentioned above.

For the Cauchy problems of scalar or systems of conservation laws, classical works established $L^\infty$ or $L^1$ stability in one and multiple dimensions via methods such as maximum principles, anti-derivative techniques, energy estimates, and spectral approaches (\cite{Ilin.O1960,Glimm.1965,Glimm.L1970,Liu.1977a,Liu.1977b,Goodm.1989,Kawas.M1985,Matsu.N1994a,Freis.S1998,Hoff.Z2000,Bress.2000,Liu.Z2009a}).
Recent developments include stability under $L^2$ perturbations, studied by Kang and Vasseur \cite{Kang.V2017} and Kang, Vasseur and Wang \cite{Kang.V.W2019}, stability under space-periodic perturbations, studied by Xin, Yuan and Yuan \cite{Xin.Y.Y2019} and Yuan \cite{Yuan.2023}, as well as shock structures in the presence of space-periodic fluxes were investigated by Dalibard \cite{Dalib.2010}, and Dalibard and Kang \cite{Dalib.K2017}; see these works and the references therein.

However, in the IBVP, boundary effects on the shock dynamics must be taken into account, and the boundary value $u_b(t)$ in \eqref{bc0} may be incompatible with the values of $u^S$ and $\phi^S$ on $x=0$, i.e., 
$$u_b(t)-u^S|_{x=0}\neq0, \quad u_b(t)-\phi^S|_{x=0}\neq0.$$
When $u_b(t)$ remains as a constant, these effects only appear in the viscous case. In \cite{Liu.N1997}, Liu and Nishihara proved the stability of the shock when $s<0$, where asymptotic state becomes $\phi^S(x-st-X(t))$ and $X(t)$ approaches a constant depending on the initial and boundary data. Moreover, stability of the stationary solution when $s>0$ was also obtained. The result for $s<0$ was extended to the hyperbolic system by Deng, Wang and Yu in \cite{Deng.W.Y2012}.
When $s=0$, the constant boundary forces the stationary shock to propagate away from the boundary, i.e. the shift $X(t)=\mathcal{O}(\ln t)$ (\cite{Liu.Y1997a,Nishi.2001}), a phenomenon regarded as a \emph{drifting effect of resonance} in \cite{Liu.2021a}. The behaviors with boundary corresponding to rarefaction waves were studies in Liu, Matsumura and Nishihara \cite{Liu.M.N1998}. 
There are also extensive works devoted to studying the stability of shocks, rarefaction waves, and stationary solutions for IBVPs of the Navier-Stokes equations with constant boundaries; see \cite{Matsu.M1999,Nishi.Y1997,Matsu.N2001,Kawas.N.Z2003,Huang.M.S2003,Kawas.Z2009,Hong.W2017,Wang.W2025,Han.K.K.K.O2025,Kang.O.W2025} and and the references therein.

For time-periodic boundaries, which arise from as the classical piston problem in fluid dynamics, existing studies have mainly focused on the existence and stability of time-periodic solutions themselves. 
In the viscous case,  Matsumuara and Nishida \cite{Matsu.N1989} considered the piston problem of viscous gas in a bounded interval, and proved the existence of the time-periodic solutions to the one-dimensional compressible Navier-Stokes generated by time-periodic boundaries or forces. 
It was proved by Kagei and Oomachi \cite{Kagei.O2016} that the time-periodic solution for the half space is asymptotically stable when the Reynolds number is sufficiently small.
Feireisl and his collaborators \cite{Feire.G.S2023} showed the existence of a time periodic weak solution to Navier–Stokes–Fourier system subject to time-periodic heating boundary and potential force.
In the inviscid case, Yuan \cite{Yuan.2019} obtained existence of time-periodic supersonic solutions in a finite interval to the one space-dimensional isentropic compressible Euler equations subjected to periodic boundary conditions.
For quasilinear hyperbolic systems with time-periodic boundaries possessing a dissipative structure, the existence, uniqueness and stability of the time-periodic classical solutions were proved in \cite{Qu.2020}. 
We refer the interested readers to \cite{Duan.U.Y.Z2007,Beck.S.Z2010,Jin.Y2015,Kagei.T2015a,Tsuda.2016,Qu.Y.Z2023} for more works on time-periodic solutions for fluid dynamics. 

The results reviewed above are mostly concerned with the long-time stability of prescribed background wave patterns, treating the incompatibility of boundary values merely as a perturbation. In contrast, Liu introduced the concept of \emph{boundary waves} to treat the boundary Riemann problem for scalar conservation laws \cite{Liu.2021a}. Given the prescribed boundary state and the interior trace, the corresponding Riemann problem is resolved in terms of elementary waves, together with an admissible boundary discontinuity when necessary; in this way, the wave pattern generated by the boundary is determined. The boundary Riemann problem thus provides a natural starting point for the stability analysis, where one should further determine how the boundary waves propagate, interact with, and modify the underlying wave pattern. 
For time-periodic boundary data, this issue is particularly delicate, since perturbations are generated repeatedly over infinitely many periods and \emph{their cumulative interaction with a shock may potentially lead to a resonance-type mechanism}. 
For example, Majda and Rosales \cite{Majda.R1984} proved the existence of resonantly interacting waves for full compressible Euler equations by studying the periodic solution.


Despite the importance of time-periodic boundary problems,
the fundamental roles of nonlinear coupling and potential resonance in shaping the long-time dynamics of Riemann solutions have remained open and challenging. This work addresses this gap by providing a rigorous analysis of time-periodic boundary effects on Riemann solutions.

\vline

In this paper, we study the cases that the shocks move away  from the boundary ($s\leq0$ in \eqref{RH}) and stationary shocks ($s=0$) under the incoming boundary condition: 
\begin{equation}
	\label{bc-incoming}
	\tag{InC}
	f'(u_b(t))<-\delta_b<0 \qquad \text{for all~} t>0  \text{ and some constant } \delta_b>0.
\end{equation}
Although the boundary continuously generates oscillatory waves over infinitely many periods, our analysis shows that for scalar conservation laws \eqref{eqn-0}, the resulting large-time dynamics are \emph{non-resonant} in the following sense: the repeated boundary wave does not produce an unbounded amplification of the shock displacement or a new growing oscillatory mode. Instead, the periodic boundary wave enters the asymptotic state through a persistent time-periodic component and contributes a finite cumulative correction or modulation to the final shock location.


For the inviscid case with $s<0$, we provide a detailed characterization of this mechanism. After a finite time, the time-periodic solution induced by the boundary occupies the whole region between the physical boundary and the perturbed shock. In other words, time-periodic boundary effects directly act on the shock curve without any transitional region between the shock and time-periodic boundary solution.
The shock curve approaches $st+X_\infty$ with an explicitly identified phase correction. In particular, although the excessive flux over each complete temporal period vanishes, the oscillatory boundary wave leaves a nontrivial cumulative contribution to $X_\infty$. Thus the persistent wave modifies the shock phase without producing resonance.

The case of inviscid stationary shock ($s=0$) has a different asymptotic geometry. after a finite time there exists a perturbed shock curve $X(t)$ that separates the left state from the time-periodic solution induced by the boundary. However, the asymptotic position $X(t)$ is no longer constant. Instead there exists a $T$-periodic curve $X_{\rm per}(t)<0$ such that $
X(t)-X_{\rm per}(t)\to0.$
The periodic curve is determined by a mass-balance identity involving the time-periodic boundary state. Hence, the zero-speed inviscid problem exhibits a persistent periodic modulation of the shock.

For the viscous problem we prove the stability theorem only for shocks with $s<0$. The viscous effect induces a transitional region where the viscous shock profile couples with the time-periodic solution, and the persistent periodic boundary state are genuinely coupled.
We propose an ansatz capturing the propagation of the viscous shock and the boundary wave, as well as their transitions.
The ansatz plays an important role in the anti-derivative argument for stability analyses of Riemann solutions in previous works.
In this paper, the construction of the ansatz is inspired by the above observations in the inviscid case, and experiences of our previous works on the problems of spatial periodic perturbations \cite{Xin.Y.Y2019,Xin.Y.Y2021}.  
It is proved that the solution asymptotically converges to the superposition of the shifted background viscous shock and the time-periodic boundary wave. 
The viscous stationary case is not treated by a complete theorem here; we present a discussion in Remark \ref{rem-viscous0} for the drifting effect found in \cite{Liu.Y1997a,Nishi.2001} in distinguishing this regime from the inviscid case.

This paper is organized as follows. In Section 2 we first present the notations, and the theorems. In Section 3, we prepare some preliminary lemmas, including generalized characteristics, properties of viscous shocks and time-periodic solutions. Then the stabilities of inviscid and viscous shock waves under time-periodic boundaries are proved in the subsequent two sections, respectively.

\section{Main results}

	\textit{Notations.} In the  paper, $L^p(U)$ and $H^{k}(U)$ ($1\leq p\leq \infty$) denote the classical Lebesgue space, Sobolev space, respectively. Without indications on the sets, the spaces are regarded as the ones of the functions on $\R_-$ and $\norm{\cdot}$ denotes for $L^2$-norm on $\R_-$. 
	    
    In the viscous case, we also consider the weighted Sobolev spaces: given weight function $w(x)$, $L^2_w$ denotes the space of measurable functions $f$ satisfying $\sqrt{w}f\in L^2$ with the norm 
	$$|f|_w=\norm{\sqrt{w}f}
	.$$
	When $w(x)=\langle x \rangle^\beta = (1+x^2)^{\beta/2} $, we also write $L^2_\beta:=L^2_w$ and $|\cdot|_\beta:=|\cdot|_w$ for simplicity.
	$H^l_\beta$ denotes the space of measurable functions $f$ satisfying $\p_x^k f\in L^2_{\beta+k}~(k=0,1,\ldots,l)$  with the norm 
	$$\norm{f}_{H^l_\beta}=\left(\sum_{k=0}^{l} |\p_x^k f|_{\beta+k}^2 \right)^{\frac{1}{2}}.$$
	
	For periodic functions $v(t)$ on $[0,T]$, denote $\hat{v}_k$ are the Fourier coefficients, i.e., 
	$$\hat{v}_k=\frac{1}{T}\int_{0}^{T} v(t) e^{-i\frac{2k\pi}{T} t} dt, \andd v(t)= \sum_{k\in\mathbb{Z}} \hat{v}_k e^{i\frac{2k\pi}{T} t}.$$
	For $1\leq p\leq \infty$ and $m\geq0$, the norms $\norm{\hat{v}}_{l^p}$, 
	$\norm{v}_{H^m_{\mathrm{per}}(0,T)}$
	is defined by 
	\begin{align*}
		&
		\norm{\hat{v}}_{l^p}= \left( \sum\limits_{k\in\mathbb{Z}}  \abs{\hat{v}_k}^p  \right)^{\frac{1}{p}},
		&&\norm{v}_{H^m_{\mathrm{per}}(0,T)}= \left( \sum\limits_{k\in\mathbb{Z}} (1+ (\frac{2k\pi}{T})^{2} )^m\abs{\hat{v}_k}^2  \right)^{\frac{1}{2}}.
	\end{align*}
	We use $c$ or $C$ to represent the generic (small or large) positive constant, and $f(x)\sim g(x)$ as $x\rightarrow a$  to represent $C^{-1}g \leq f\leq C g$  in a neighborhood of $a$.

\subsection{Main results for the inviscid shocks} 

Before presenting the result, the time-periodic solution to the boundary value problem (BVP) \eqref{eqn-0} and \eqref{bc0} is needed.
\begin{lemma}[Inviscid time-periodic solutions]
	\label{lem-per-t0}
	Assume that $f$ is strictly convex, the boundary data $ u_b \in L^\infty(0,\infty) $ is periodic in $t$ with period $ T>0 $ and satisfies the incoming boundary condition \eqref{bc-incoming}.
	
	Then there exists an unique time-periodic entropy solution $ \ur(x,t) $ with period $T$ to the boundary value problem \eqref{eqn-0} (for $\mu=0$) and \eqref{bc0}. Moreover, $\ur(x,t)$ satisfies that
	\begin{equation}
		\label{ineq-per-t0-1}
		 \int_{0}^{T} f(\ur(x,t))dt = \int_{0}^{T} f(u_b(t))dt, ~~
		\| \ur(x,\cdot) - \ubr \|_{L^\infty(\R)} \leq \frac{C}{|x|}, 
	\end{equation}
	for all $x<0$, where $\ubr$ is the unique state satisfying $f'(\ubr)<0$ and
	 $f(\ubr)	= T^{-1} \int_{0}^{T} f(u_b)dt $,
	and $ C>0 $ is a constant independent of $ x $.
	In addition, for any point $t_b\in[0,p)$ such that
	\begin{equation}
		\label{ineq-per-t0-2}
		\int^{t_b}_0 [f(u_b(\tau))-f(\ubr)] d\tau=\max\limits_{t\in[0,T]} \int^t_0 [ f(u_b(\tau))-f(\ubr)] d\tau,
	\end{equation}
	and for each $N\in \mathbb{N}$, $u$	takes the constant value $ \ubr $ along the straight line $ x = f'(\ubr) (t -NT-t_b) $.
\end{lemma}

This lemma can be proved by interchanging the role of $t$ and $x$; see Lemma \ref{lem-interchange x&t}. To our best knowledge, this method can be traced back to \cite{Li.Y2017}, and we also refer interested readers to \cite{Yuan.2019} for more details. The complete proof of Lemma \ref{lem-per-t0} is given in \cref{sect-time-peri0}. 

\vline

The main result for the inviscid shock is stated as follows: 
\begin{theorem}\label{thm-shock}
		 Assume that $\mu=0$, $f\in C^2$ is strictly convex, $ \ubl>\ubr$, \eqref{RH} with $s\leq0$.  
		 Assume further that the boundary value $ u_b \in L^\infty(0,\infty) $ is periodic with period $T$, and satisfies \eqref{bc-incoming}, and that the initial data $u_0$ satisfies $u_0-\ubl \in L^\infty\cap L^1(\R_-)$, and there exist $X_I<0$ such that
		 \begin{equation}
		 	f'(u_0(x))<0,\qquad \text{for~a.e.~} X_I<x<0, 
		 	\label{ic-incominga}
		 \end{equation}
		 and
		 \begin{equation}
		 	 \sup_{x\le X_I} \int_x^0\bigl(u_0(y)-\ubl \bigr)\, dy 
		 	<
		 	\min_{t\in[0,T]} \int_0^t [f(u_b(\tau))-f(\ubr) ]\, d\tau.
		 	\label{ic-incomingb}
		 \end{equation} 
	 Then there exist a Lipschitz continuous curve $ X(t)\in \mathrm{Lip}(0, +\infty), $ which is unique after a finite time $ T_S>0$, such that the entropy solution $ u(x,t) $ to \eqref{eqn-0}, \eqref{bc0}, \eqref{ic0} satisfies that for all $ t>T_S, $
	\begin{equation}
		\label{ineq-thm-shock1}
		u(x,t) = 
		\begin{cases}
			\ubl+\mathcal{O}(1)t^{-1/2} & \text{ if } x<X(t), \\
			\ur(x,t)=\ubr+\mathcal{O}(1)x^{-1} & \text{ if } x>X(t).
		\end{cases}
	\end{equation}
	where $\ur(x,t)$ is the time-periodic entropy solutions to BVP \eqref{eqn-0} and \eqref{bc0} obtained in Lemma \ref{lem-per-t0}.
	\begin{enumerate}[label = \rm (\roman*),ref = \rm(\roman*)]
		\item If $s<0$, there exists a constant $ C>0, $ independent of time $ t, $ such that 
		\begin{equation}\label{ineq-shift}
			\abs{X(t)-st-X_\infty} \leq \frac{C}{\sqrt{t}}, 
		\end{equation}
		with the constant shift $ X_\infty $ given by
		\begin{equation}\label{eq-finalshift}
			\begin{aligned}
				X_\infty := \frac{1}{\jump{\ub}} \Big\{ -& \int_{-\infty}^{0} \left( u_0(y)-\ubl\right)  dy  + \max\limits_{t\in[0,T]} \int^t_0 [ f(u_b(\tau))-f(\ubr)] d\tau \Big\}.
			\end{aligned}
		\end{equation}
		
		\item If $s=0$, then there exists a unique $T$-periodic curve $X_{\rm per}(t)<0$ such that
		\begin{equation}
			X(t)-X_{\rm per}(t)\longrightarrow0
			\qquad\text{as }t\to\infty.
			\label{eq:Xconv-per}
		\end{equation}
	\end{enumerate}
	
\end{theorem}

\begin{remark}\label{rem-boundary wave}
	After a sufficiently large time, $X(t)$ is exactly the perturbed shock curve. In fact, 
	\eqref{ineq-thm-shock1} shows that $u(X(t)-,t)$ tends to $\ubl$ as $t$ tends to infinity, while for $s<0$,  \eqref{ineq-shift} implies that $u(X(t)+,t)$ tends to $\ubr$, and for $s=0$, convexity of $f$ and \eqref{bc-incoming} ensure that $\ur <u_*$, a constant strictly smaller than $\ubl$. 
	
	This shows that the time-periodic state generated from the boundary propagates all the way to the shock curve, with no intermediate transition region. In this sense the persistent boundary effect is carried by a propagating boundary wave rather than a localized boundary layer. This boundary wave remains time-periodic and does not decay in time.
\end{remark}

\begin{remark}\label{rem-effects}
	For $s<0$, the time-periodic boundary wave has a cumulative effect on the shock location rather than producing a resonant amplification, i.e. second term in the formula for $X_\infty$. By the definition of $\ub_+$, there is no excessive flux on every single temporal period. However, the accumulation of an infinite number of perturbations produces a nontrivial shift, which is also observed in the previous works on the problems of spatial periodic perturbations (\cite{Xin.Y.Y2019,Xin.Y.Y2021}).
	
   For $s=0$, the effect is different: the time-periodic boundary wave produces a persistent periodic modulation of the shock, as shown by \eqref{ineq-thm-shock1} and \eqref{eq:Xconv-per}. More precisely, as $t\rightarrow$, $u \rightarrow \ubl$ if $x<X(t)$ and $u = \ur$ if $x>X(t)$, and $X_{\rm per}$ is exactly the shock wave curve between $\ubl$ and $u = \ur$; see Remark \ref{rem-Xper} below. 
	
\end{remark}

\begin{remark}
The incoming boundary condition \eqref{bc-incoming} plays an important role in enabling the method of interchanging the role of $t$ and $x$ in Lemma \ref{lem-per-t0}, as that in \cite{Li.Y2017,Yuan.2019}.

\eqref{ic-incominga} and \eqref{ic-incomingb} are just imposed to just avoid the intersections and reflections of the forward characteristics with the boundary.
More precisely, \eqref{ic-incominga} excludes outgoing characteristics near the boundary, while \eqref{ic-incomingb} provides a cumulative condition which prevents possible outgoing characteristics generated farther to the left from reaching the boundary; see Lemma \ref{lem-hit} below for details.
When $s<0$, a simpler sufficient assumption can be imposed 
\begin{equation}
	\label{ic-incoming1}
	f'(u_0(x))<0, \qquad \text{for~}x\in\R_-.
\end{equation}
\eqref{ic-incominga} and \eqref{ic-incomingb} can be replaced or removed in some other special cases; for instance, it can be replaced by the smallness assumption on the initial perturbation around the shock.  
\end{remark}

\subsection{Main results for the viscous shock profiles} 

We first introduce the viscous shock and the boundary wave (i.e. the time-periodic solution to the boundary value problem (BVP) \eqref{eqn-0} and \eqref{bc0}), and then present the lemma and the theorem.
 
\begin{lemma}[\cite{Liu.N1997}]\label{lem-shock-prop}
	Assume that the Rankine-Hugoniot condition \eqref{RH} and the Oleinik entropy condition
	\begin{equation}
		\label{E} \tag{E}
		f_0(\phi):=-s(\phi-\ub_\pm)+f(\phi)-f(\ub_\pm)\begin{cases}
			<0 & \text{if~~} \ubr<\phi<\ubl,\\
			>0 & \text{if~~} \ubl<\phi<\ubr,
		\end{cases}
	\end{equation}
	hold. Then \eqref{eqn-0} (with $\mu=1$) admits a unique traveling wave solution (up to a shift) $\phi^S (x-st)$ satisfying 
	\begin{equation}
		\label{eqn-shock1}
		(\phi^S)'(\xi)=f_0(\phi^S(\xi)),\quad \lim_{\xi\rightarrow\pm\infty}\phi^S(\xi)=\ub_\pm.
	\end{equation}
	Moreover, for \emph{non-degenerate} shock, i.e. \eqref{non-degenerate-shock}, the shock profile $\phi^S$ satisfies
	\begin{equation}\label{ineq-shock1}
		\abs{ \p_\xi^k (\phi^S(\xi)-\ub_\pm)} 
		\sim  e^{-\theta_s \abs{ \xi} }
		\qquad \text{as~}\xi \rightarrow \pm \infty
	\end{equation}
	for $k=0,1$ and $\theta_s$ is a positive constant.
\end{lemma}
When $s=f'(\ub_\pm)$, the shock is called \emph{degenerate}, and \eqref{ineq-shock1} is quite different. We refer the reader to \cite{Matsu.N1994a} for the detailed proof.

\vline

The solution space $\mathcal{T}^m_{\gamma,M}$ for the time-periodic solutions is
\begin{align*}
	\mathcal{T}^m_{\gamma,M}= \big\{ u&\in C(\R_-; H^m_{\mathrm{per}}(0,T)); ~\exists \ub \in \R \text{ s.t. for integers~} i, j ~\text{~satisfying~}
	\\
	& 0\leq i+\frac{j}{2}\leq m, \text{there holds~}  \sup_{x\in\R_-} \{e^{-\gamma x} \norm{ \p_t^i\p_x^j (u-\ub) }_{H^{m-i-\frac{j}{2}}_{\mathrm{per}}(0,T)}\} <M  \big\}.
\end{align*}
for constants $m\geq1, \gamma>0, M>0$.

\begin{lemma}[Viscous time-periodic solutions]
	\label{lem-per-t1}
	Assume that $f\in C^2$, and $ u_b\in$ $ H^1_{\mathrm{per}}(0,T) $ is periodic with period $T$ and average $ \ub_b = \frac{1}{T} \int_{0}^{T} u_b(t)dt$, and satisfies the incoming boundary condition \eqref{bc-incoming}.
	There exist positive constants $\theta_b,$ $\nu_b$ and $C_0$ such that if 
	$$\norm{ u_b-\ub_b }_{H^1_{\mathrm{per}}(0,T)}\leq \nu_b,$$
	 then the boundary value problem \eqref{eqn-0} (with $\mu=1$) and \eqref{bc0}  admits a unique time-periodic solution $\ur\in 
	\mathcal{T}^1_{\theta_b, C_0\nu_b}$ of period $T$. 
	Moreover, for $m\in\mathbb N_+$, $u_b\in H^m_{\mathrm{per}}(0,T)$ and $f\in C^{m+1}$,
	there exists a constant $\ubr$ such that
	\begin{equation}
		\label{ineq-per-t1}
		\sum\limits_{0\leq i+j/2\leq m}\norm{\p_t^i\p_x^j (\ur-\ubr)  }_{
			{H}^{m-i-\frac{j}{2}}_{\mathrm{per}}(0,T)} 
		\leq C_0 e^{ \theta_b x} \norm{ u_b -\ub_b  }_{H^m_{\mathrm{per}}(0,T)}.
	\end{equation}
\end{lemma} 
This lemma cannot be directly derived by  interchanging the role of $x$ and $t$ due to the viscous term. In contrast, the proof is proved by Fourier transform with respect to time variable and a suitable iteration argument.
For Navier--Stokes equations, more complicated techniques are needed to study of time-periodic solutions; see \cite{Brezi.K2012,Feire.M.P.S1999,Feire.G.S2023,Jin.Y2015,Kagei.T2015a,Ma.U.Y2010,Tsuda.2016} and the references therein. The proof is provided in \cref{sect-time-peri1}. Note that The far-field state $\ubr$ is determined implicitly by the viscous time-periodic solution, unlike the inviscid case. In particular,
$$
f(\ubr)=\frac1T\int_0^T
\bigl[f(u_b(t))-\p_x\ur(0,t)\bigr]\,dt,
\quad
|\ubr-\ub_b|\le
C\norm{u_b-\ub_b}_{H^1_{\rm per}(0,T)}^2.
$$

\vline

Now we construct the \emph{ansatz} $\usharp$, which plays an important role in the anti-derivative argument for stability analyses of Riemann solutions in previous works.
Technically, $\usharp$ is set such that $\int_{-\infty}^{0} u(x,t)-\usharp(x,t) dx=0$ and $\abs{u-\usharp}\rightarrow 0 $ as $t\rightarrow\infty$. Therefore,  $\usharp$
has to capture the propagation of the shock profile $\phi^S(\xi)$ and the time-periodic boundary wave $\ur$ obtained in Lemma \ref{lem-per-t1}, as well as their transitions.

Note that $\phi^S$ can be decomposed as 
\begin{equation*}
	\phi^S(\xi)=\ubl(1-\sigma(\xi))+\ubr\sigma(\xi),\quad
\end{equation*}
where $\sigma$ is the weighted function induced by $\phi^S$, i.e., 
\begin{equation*}
	\label{def-weight}
	\sigma(\xi):=\dfrac{\phi^S(\xi)-\ubl}{\ubr-\ubl}=\dfrac{f(\phi^S(\xi))-f(\ubl)-\p_\xi \phi^S(\xi)}{f(\ubr)-f(\ubl)}.
\end{equation*} 
Given $\ur$ obtained in Lemma \ref{lem-per-t1}, and inspired by the observations in the inviscid case and \cite{Xin.Y.Y2021}, we construct the ansatz as
\begin{equation}
	\label{def-ansatz-shock}
	\usharp:=\ubl (1-\sigma_{\X}) + \ur \sigma_{\X} = \phi^S_{\X} + (\ur-\ubr) \sigma_{\X}.
\end{equation}
where $\X=\X(t)$ is the shift function to be determined, and
we adopt the convention that a subscript $\X$ denotes a spatial shift by $st + \X$: for any function $g(\xi)$,
\begin{equation*}
	\label{def-convention-X}
	g_{\X}(x,t)=g(x-st-\X).
\end{equation*}
Subscripts $x$ or $t$, however, still denote partial differentiation with respect to $x$ or $t$.

Consequently, the anti-derivative is defined as 
\begin{equation}
	\label{def-anti}
	U(x,t)=\int_{-\infty}^x (u-\usharp)(y,t) dy \,, \quad U_0(x,t)=\int_{-\infty}^x (u-\usharp)(y,0) dy.
\end{equation}



\begin{theorem}
	\label{thm-shock-viscous}
	Assume $\mu=1$, $f\in C^4$ and that $ u_b\in$ $ H^1_{\mathrm{per}}(0,T) $ is periodic with period $T$ and average $ \ub_b$, and that $\ur(x,t)$ is the time-periodic entropy solution in Lemma \ref{lem-per-t1} with $\ur(x,t)\rightarrow \ubr $ as 	$x\rightarrow-\infty$. Furthermore, assume \eqref{RH} with $s<0$, \eqref{E}, and that the shock is non-degenerate, i.e. \eqref{non-degenerate-shock}
	. 
	
	Then there exists $ \varepsilon_0>0$ such that if
	\begin{equation}
		\label{small-perturbation}
		\norm{U_0}_{H_{\beta}^3}+|\X_0|^{-1} + \norm{u_b(t)-\ub_b}_{H^3_{\mathrm{per}}(0,T)} < \varepsilon_0,\quad 
	\end{equation}  
	for $\beta>0$, 
	then the problem \eqref{eqn-0}, \eqref{bc0} and \eqref{ic0} admits a unique global solution $ u(x,t) $ satisfying
	\begin{align*}
		&u(x,t)-\usharp(x,t)\in C([0,\infty); H_{\beta+1}^2) \cap  L^2(0,\infty; H_{\beta}^3),\\
		&\p_x u(0,t)- \p_x \usharp (0,t) \in C[0,\infty) \cap  L^2(0,\infty), 
	\end{align*}
	and the asymptotic behavior
	\begin{equation}
		\label{asymptotic behavior}
		\sup_{x\in\R_-} |u(x,t)-(\phi^S_{\X_{\infty}}+\ur-\ubr)(x,t)|\rightarrow 0 \qquad \text{as~}t\rightarrow\infty
	\end{equation}	
	for some constant $\X_{\infty}$,
    where $\usharp$ is defined in \eqref{def-ansatz-shock} and the phase shift $\X(t)$ is given by 
    \begin{equation}
    	\label{eqn-shift1'}
    	\begin{aligned}
    		\Big(\int_{-\infty}^0 \jump{u} \p_x \sigma_{\X} dx\Big)&  \X'(t) = \Big( f(\ubr)-f(\phi^S_{\X})\big|_{x=0}-\p_x u(0,t) + \p_x \usharp (0,t) \Big)   \\
    		& +\int_{-\infty}^0 \{ - (\ur - \ubr) s+ f(\ur)-f(\ubr)-\p_x \ur \} \p_x\sigma_{\X} dx\\
    		& +\Big( (f(u_b)-f(\ubr))(1-\sigma_{\X}) - (u_b-\ubr) \p_x\sigma_{\X}  \Big)\big|_{x=0} \,	,
    			\\
    			&\X(0)=\X_0.
    	\end{aligned}
    \end{equation}    	
\end{theorem} 

We note that in the viscous result $f$ is not necessarily convex. From \eqref{RH} and \eqref{E} we have $f'(\ubr)\leq s\leq f'(\ubl)$,  and non-degeneracy of the shock just requires that $s\neq f'(\ub_\pm)$. With the aid of $s<0$ and \eqref{small-perturbation}, the incoming boundary condition \eqref{bc-incoming} holds.

\begin{remark}
	\eqref{asymptotic behavior} shows that the asymptotic state is the superposition of the shifted background viscous shock $\phi^S_{\X_{\infty}}$ and the time-periodic boundary wave $\ur$. 
	Such a boundary wave keeps oscillating as time goes to infinity.
\end{remark}

\begin{remark}
	Note that the weight of boundary wave $\ur-\ubr$ in $\usharp$ \eqref{def-ansatz-shock}, i.e. $\sigma$, tends to $0$ as $x\rightarrow -\infty$. 
	It implies that our ansatz $\usharp$, rather than the linear superposition $\ur - \ubr + \phi^S_{\X}$, captures the transitions of the boundary wave, and it aligns with the inviscid case observation that the boundary wave does not appear on the left region of the shock curve. 
	
	As time approaches, $\usharp$ converges to the linear superposition $\ur - \ubr + \phi^S_{\X}$, since the nonlinear interactions of the viscous shock and the boundary wave would weaken progressively as the viscous shock propagates away from the boundary; see \eqref{ineq-two ansatz} below.  
\end{remark}

	

\begin{remark} \label{rem-viscous0}
	The viscous stationary case $s=0$ is not considered here. In this case, viscosity may induce a logarithm drift of the shock even for constant boundary data as shown by \cite{Liu.Y1997a,Nishi.2001}. With a time-periodic boundary, the interaction between the boundary wave	and the viscous shock is expected to become weaker as the viscous shock propagates away from the boundary. Therefore, the large-time behavior is expected to be close to their linear superposition, as the case $s<0$ of Theorem \ref{thm-shock-viscous}. We leave this for future study.
\end{remark}

\section{Preliminaries and time-periodic solutions}

Before proving the theorem, in this section we present some preliminary lemmas and complete the proof of Lemmas \ref{lem-per-t0} and \ref{lem-per-t1}.

\subsection{Generalized characteristics}

It is well-known in the theory of generalized characteristics that for one-dimensional scalar convex conservation laws, through every point $(x,t)\in \R \times[0,+\infty)$ pass two forward (backward) extreme generalized characteristics, which may coincide at some time. Moreover, when $t>0,$ the forward generalized characteristic is \emph{unique} and the two backward extreme generalized characteristics are both straight lines. We refer to \cite[Chapters 10 \& 11]{Dafer.2016} for details, or \cite[Definition 3.1 \& Lemma 3.2]{Xin.Y.Y2019} for a short summary. 

Thanks to nice properties of generalized characteristics of convex scalar conservation laws, the lemmas as follows can be derived.

\begin{lemma}[Space-periodic solutions]\cite[Theorem~5.2]{Glimm.L1970}
	\label{lem-per-x}
	Assume that the initial data $ u_0 \in L^\infty $ is periodic in $x\in\R$ with period $ p>0 $ and average $ \ub := \frac{1}{p} \int_{0}^{p} u_0(x)dx. $ Then the entropy solution $ u(x,t) $ to the Cauchy problem \eqref{eqn-0} and \eqref{ic0}, for $\mu=0$, $x\in \R, t\geq 0$, is also periodic with period $ p>0 $ and average $ \ub $ for each $ t\geq 0, $ and satisfies that
	\begin{equation}
		\| u(\cdot,t) - \ub \|_{L^\infty(\R)} \leq \frac{C}{t}, \quad t>0,
	\end{equation}
	where $ C>0 $ is a constant, independent of time $ t. $
\end{lemma}

\begin{lemma}\cite[Proposition 3.2]{Dafer.1995}\label{lem-div}
	Assume that the initial data $ u_0 \in L^\infty(\R) $ in \eqref{ic0}. Then the entropy solution $ u(x,t) $ to the Cauchy problem \eqref{eqn-0} and \eqref{ic0} for $\mu=0$ takes a constant value $ \ub $ along the straight line $ x = f'(\ub) t + x_0 $ for some $ x_0\in\R, $ if and only if 
	\begin{equation}
		\int_{x_0}^x (u_0(y)-\ub) dy \geq 0 \quad \forall x\in\R.
	\end{equation}
\end{lemma}
It is noted that such a straight line is a classical characteristic and is called a \emph{divide} by C. Dafermos in \cite{Dafer.1995} (also see \cite[Definition 10.3.3]{Dafer.2016}). It is obvious that there exists at least a divide issuing from any period for any space-periodic solution.

\begin{lemma}\cite[Theorem 11.5.1]{Dafer.2016}\label{lem-lr-decay}
	Let $u$ be the admissible solution to \eqref{eqn-0} for $\mu=0$ with initial data $u_0$ such that 
	\begin{equation}
		\int_{x}^{x+l}u_0(y)dy =\mathcal{O}(l^{r}), \qquad\text{as~}l\rightarrow\infty,
	\end{equation}
	for some $r\in[0,1)$, uniformly in $x$ on $\R$. Then 
	\begin{equation}
		u(x\pm,t)=\mathcal{O}(t^{-\frac{1-r}{2-r}}), \qquad\text{as~}t\rightarrow\infty,
	\end{equation}
	uniformly in $x$ on $\R$.
\end{lemma}
As a direct application of this lemma, if $u_0\in L^p{\R}$, $1\leq p<\infty$, then it holds that
\begin{equation}
	u(x\pm,t)=\mathcal{O}(t^{-\frac{p}{p+1}}), \qquad\text{as~}t\rightarrow\infty,
\end{equation}
uniformly in $x$ on $\R$.

\subsection{Inviscid time-periodic solutions} \label{sect-time-peri0} 

To study the inviscid time-periodic solution, observe that the BVP \eqref{eqn-0} and \eqref{bc0} with incoming boundary \eqref{bc-incoming} is equivalent to a Cauchy problem. And then the proof of Lemma \ref{lem-per-t0} follows.

\begin{lemma}\label{lem-interchange x&t}
	Assume that incoming boundary condition \eqref{bc-incoming} holds. Then $u$ is the entropy solution to boundary value problem \eqref{eqn-0} and \eqref{bc0} (extended to $t\in \R$) if and only if $v(\tilde{x},\tilde{t}):=-f(u(-\tilde{t},\tilde{x}))$ is the entropy solution to
	\begin{equation}
		\label{eqn-inco}
		\begin{aligned}
			\p_{\tilde{t}} v(\tilde{x},\tilde{t}) +\p_{\tilde{x}} g(v(\tilde{x},\tilde{t})) &=0
			\quad \tilde{x} \in \R, \tilde{t}>0,  \\
			v(\tilde{x},0) & = v_0(\tilde{x}), \quad \tilde{x}\in \R ,
		\end{aligned}
	\end{equation}
	where $v_0(\tilde{x}):=-f(u_b(\tilde{x}))$, $g(v):=(f^{-1})(-v)$ for $v\in [\min v_0, \max v_0]$, where $f^{-1}$ denotes the inverse.
\end{lemma}

We note that thanks to \eqref{bc-incoming} and the convexity of $f$, $g$ is well-defined and strictly convex. Moreover, $v_0\in L^{\infty}(\R)$ and
\begin{equation}
	\label{bc-incoming2}
	g'(v_0)= -(f'(u_b))^{-1}<\delta_b^{-1}. 
\end{equation}
It is straightforward to verify the equation of $v$, Rankine-Hugoniot condition and Lax entropy condition, and thus the proof of this lemma is omitted. We refer the interested readers to \cite[p.14]{Li.Y2017} and \cite{Yuan.2019} for more details on the method of interchanging the role of $x$ and $t$. 

\begin{proof}[Proof of Lemma \ref{lem-per-t0}]
	After interchanging the role of $x$ and $t$ as Lemma \ref{lem-interchange x&t}, \eqref{ineq-per-t0-1} can be obtained directly by Lemma \ref{lem-per-x}. The existence of divides follows from Lemma \ref{lem-div}, and in the equivalent Cauchy problem the initial point of the divides, $\tilde{x}_0$, is chosen such that 
	$$
	\int^{\tilde{x}_0}_0 [v_0(y)-\bar{v}] dy=\min\limits_{\tilde{x}\in[0,p]} \int_0^{\tilde{x}_0} [ v_0(y)-\bar{v}] dy,
	$$
	where $\bar{v}=p^{-1}\int_{0}^{p} v_0(y) dy$. It is equivalent to \eqref{ineq-per-t0-2}.
	Therefore, Lemma \ref{lem-per-t0} is proved.
\end{proof}

\subsection{Viscous time-periodic solutions}
\label{sect-time-peri1} 
Now we prove the existence and decay estimates of viscous time-periodic solutions stated in Lemma \ref{lem-per-t1}. 
\begin{proof}[Proof of Lemma \ref{lem-per-t1}]
	\underline{Step 1. Fourier transform and the linearized equation}.
	Let $v:=u-\ub_b,~  v_b(t)=u_b(t)-\ub_b,~ f_1(v):=f'(\ub_b+v)-f'(\ub_b).$ 
	$v$ satisfies 
	\begin{equation}
		\label{eqn-per-time}
		\begin{aligned}
			\p_x^2 v -f'(\ub_b) \p_x v -\p_t v &= f_1(v)\p_x v &&\quad \text{for~} x\in \R_-, t>0,\\
			v(0,t)&=v_b(t), 
			&&\quad \text{for~}  t\geq0.
		\end{aligned}
	\end{equation}
	Applying the Fourier transform with respect to $t$ on the period $[0,T]$ on the above equation, it yields that
	\begin{equation}
		\label{eqn-per-time-Fourier}
		\begin{aligned}
			\p_x^2 \hat{v}_k -f'(\ub_b) \p_x \hat{v}_k -i\frac{2k\pi}{T} \hat{v}_k &= (\widehat{f_1(v)} * \p_x\hat{v})_k \quad \text{for~} x\in \R_-,\\
			\hat{v}_k\Big|_{x=0}&=\hat{v_b}_{,k}
			\,.
		\end{aligned}
	\end{equation}
	Here ``$*$'' denotes the convolution $(\hat{v} * \hat{w})_k = \sum_{k_1+k_2=k} \hat{v}_{k_1} \hat{w}_{k_2}$.
	
	Since $f'(\ub_b)<0$, the linear differential operator of \eqref{eqn-per-time-Fourier} has two different eigenvalues
	\begin{equation}
		\label{def-eigenvalue}
		r_\pm(k) =\frac{1}{2} \left( r_1 \pm \frac{1}{\sqrt{2}} \left( \sqrt{\sqrt{r_1^4+16r_2^2 } +r_1^2 } + i \sgnk \sqrt{\sqrt{r_1^4+16r_2^2 } -r_1^2 } \right) \right), 
	\end{equation}
	where $r_1=f'(\ub_b)<0, ~r_2=\frac{2k\pi}{T}$.
	$\pm \text{Re}\, r_\pm(k) >0$ for $k\neq0$, and $r_+(0)=0, r_-(0)=r_1<0$. For simplicity, we omit $k$ in $r_\pm(k)$ when there are not confusions.
	
	Therefore, it is easy to verify that the solution $\hat{v}_k(x)$ to \eqref{eqn-per-time-Fourier} for $k\neq0$, which decays to $0$ as $x\rightarrow -\infty$, has the form 
	\begin{align}
		\hat{v}_k(x)=& e^{r_+ x} g_k - \int_{x}^{0} \frac{e^{r_+ (x-y)} }{r_+ - r_-} (\widehat{f_1(v)} * \p_x\hat{v})_k (y) dy \notag\\
		&\qquad +  \int_{-\infty}^{x} \frac{e^{r_- (x-y)} }{r_- - r_+} (\widehat{f_1(v)} * \p_x\hat{v})_k (y) dy, \label{eqn-per-time-Fourier1}\\
		\text{where~} g_k:=&\hat{v_b}_{,k} - \int_{-\infty}^{0} \frac{e^{-r_- y} }{r_- - r_+} (\widehat{f_1(v)} * \p_x\hat{v})_k (y) dy. \notag
	\end{align}
	The solution $\hat{v}_0(x)$ to \eqref{eqn-per-time-Fourier}, of which the derivatives decay to $0$ as $x\rightarrow -\infty$, satisfies   
	\begin{align}
		\hat{v}_0(x)=& \hat{v}_{0,\infty}  + \int_{-\infty}^{x}  r_1^{-1}( e^{r_1 (x-y)} -1) (\widehat{f_1(v)} * \p_x\hat{v})_0 (y) dy, \label{eqn-per-time-Fourier2}\\ 
		\text{where}~\hat{v}_{0,\infty}:=& -\int_{-\infty}^{0}  r_1^{-1}( e^{-r_1 y }-1 ) (\widehat{f_1(v)} * \p_x\hat{v})_0 (y) dy. \notag
	\end{align}  
	Here $\hat{v_b}_{,0}=0$ due to the definition of $\ub_b$ and $v_b$.
	Note that when $\lim\limits_{x\rightarrow-\infty}(\widehat{f_1(v)} * \p_x\hat{v})_k(x)=0$, 
	the two integrals of $\hat{v}_k$ ($k\neq0$) in \eqref{eqn-per-time-Fourier1} and the integral of $\hat{v}_0$ in \eqref{eqn-per-time-Fourier2} tend to $0$ as $x\rightarrow -\infty$, that is, 
	\begin{align*}
		&\lim\limits_{x\rightarrow-\infty}\hat{v}_k (x)=0 \text{~for~} k\neq 0 , \andd
		\lim\limits_{x\rightarrow-\infty}\hat{v}_0 (x)=
		\hat{v}_{0,\infty} . 
	\end{align*}
	Therefore, to solve the boundary value problem \eqref{eqn-0} and \eqref{bc0}, it equivalents to solve \eqref{eqn-per-time-Fourier1} and \eqref{eqn-per-time-Fourier2}.
	
	\noindent\underline{Step 2. Contraction mapping}.
	Denote $\theta_b=\frac{1}{2}\mathrm{Re}\,r_+(1)>0$, and hence
	\begin{equation}
		\label{ineq-eigenvalue}
		\abs{\mathrm{Re}\, r_\pm(k) -\theta_b} \geq \theta_b>0 ~\text{for } k\neq 0
		.
	\end{equation}
	For any fixed $\gamma \in (0, \theta_b]$, define the mapping $\Xi$ as follows:
	for any $v\in \mathcal{T}^m_{\gamma,M}\cap \{v(0,t)=v_b(t)\}$, $w=\Xi(v)$ satisfies that
	\begin{align*}
		\hat{w}_k(x)&= e^{r_+ x} g_k - \int_{x}^{0} \frac{e^{r_+ (x-y)} }{r_+ - r_-} \Big(\widehat{f_1(v)} * \p_x\hat{v} \Big)_k (y) dy \\
		&\qquad  +  \int_{-\infty}^{x} \frac{e^{r_- (x-y)} }{r_- - r_+} \Big(\widehat{f_1(v)} * \p_x\hat{v} \Big)_k (y) dy,\\
		\text{where~}
		g_k	& =\hat{v_b}_{,k} - \int_{-\infty}^{0} \frac{e^{-r_- y} }{r_- - r_+} \Big(\widehat{f_1(v)} * \p_x\hat{v} \Big)_k (y) dy, \quad \text{for~}k\neq0, \\
		\hat{w}_0(x)&= \hat{w}_{0,\infty}  + \int_{-\infty}^{x}  r_1^{-1}( e^{r_1 (x-y)} -1) (\widehat{f_1(v)} * \p_x\hat{v})_0 (y) dy, \\ 
		\text{where~}
		\hat{w}_{0,\infty}&= -\int_{-\infty}^{0}  r_1^{-1}( e^{-r_1 y }-1 ) (\widehat{f_1(v)} * \p_x\hat{v})_0 (y) dy. 
	\end{align*}
	In virtue of $v\in \mathcal{T}^m_{\gamma,M}$, the far field value $\hat{v}_{0,\infty}$ of $v$ satisfies 
	\begin{align*}
		\abs{\hat{v}_{0,\infty}}&\leq |\hat{v_b}_{,0}|+ \int_{-\infty}^{0} \abs{\p_x \hat{v}_0 (y)} \, dy \\
		&\leq \int_{-\infty}^{0} e^{\gamma y}  \, dy \cdot  \sup_{y\in\R_-} e^{-\gamma y}\norm{\p_x \hat{v}_0 (y)}_{L^2_{\mathrm{per}}} 
		 \leq CM,
	\end{align*}
	where we omit $(0,T)$ in $H^m_{\mathrm{per}}(0,T)$ for short when there are no confusions.
	Hence, the nonlinear term $f_1(v) \p_x v$ can be controlled as 
	\begin{align*}
		&\norm{ f_1(v) \p_x v }_{H^{m-1}_{\mathrm{per}}} 
		\leq C   \left(\norm{ f_1(v) }_{L^\infty_{\mathrm{per} } } \norm{\p_x v}_{H^{m-1}_{\mathrm{per}}} + \norm{ f_1(v) }_{W^{m-1,\infty}_{\mathrm{per} } } \norm{\p_x v}_{L^2_{\mathrm{per} } }  \right)\\
		&\qquad
		\leq C  (\norm{v-\hat{v}_{0,\infty} }_{H^1_{\mathrm{per}} }+\abs{\hat{v}_{0,\infty}})  \norm{\p_x v}_{H^{m-1}_{\mathrm{per} } } 
		+C(\norm{v-\hat{v}_{0,\infty} }_{H^{m}_{\mathrm{per}} }+\abs{\hat{v}_{0,\infty}})  \norm{\p_x v}_{L^2_{\mathrm{per} } }  \\
		& \qquad 
		\leq  C  e^{\gamma x} M^2,
	\end{align*}
	where $|f_1(v)|\leq \max_w|f''(w)|\,|v|\leq C|v|$, $C$ is the generic constant independent of $M$.
	Moreover, it follows from \eqref{def-eigenvalue} and \eqref{ineq-eigenvalue} that $\text{for } k\neq0,$
	\begin{align*}
		&\qquad\abs{\mathrm{Re}\, r_\pm(k)}\leq C |k|^{1/2}, \quad \text{~and~} \\	&\abs{\int_{x}^{0} e^{\mathrm{Re}\, r_+ (x-y)} e^{\gamma y} dy } 
		+ \abs{\int_{-\infty}^{x} e^{\mathrm{Re}\, r_- (x-y)} e^{\gamma y} dy } 
		\leq \frac{Ce^{\gamma x}}{\abs{\mathrm{Re}\, r_\pm -\theta_b} } \leq  \frac{Ce^{\gamma x}}{|k|^{1/2}}, \\
		& \abs{\int_{-\infty}^{x} (e^{r_1 (x-y)}-1) e^{\gamma y} dy } \leq Ce^{\gamma x}.
	\end{align*} 
	Then for $m\in \mathbb{N}_+$, it holds that
		\begin{align}
			&
			\sup_{x\in\mathbb R_-}\left\{ e^{ -\gamma x}\sum_{0\leq i+\frac{j}{2}\leq m}\norm{\p_t^i\p_x^j (w- \hat{w}_{0,\infty}) }_{H^{m-i-\frac{j}{2}}_{\mathrm{per}}}  \right\}
			+\abs{\hat{w}_{0,\infty}} \notag\\		
			&\qquad
			\leq C\|v_b\|_{H^m_{\rm per}}
			+C  \sup_{x\in\R_-} \{e^{ -\gamma x} 
			\norm{ f_1(v) \p_x v }_{H^{m-1}_{\mathrm{per}}} \} \notag \\
			&\qquad\leq C\|v_b\|_{H^m_{\rm per}}+CM^2. \label{ineq-per-t1m0}
		\end{align}
	Consequently, by setting $m=1$, $\gamma=\theta_b$, and $
	M=C_0\nu_b,$
	where $C_0>0$ is fixed sufficiently large, it follows from \eqref{ineq-per-t1m0} that
	\[
	C\|v_b\|_{H^1_{\rm per}}+CM^2\le M
	\]
	provided that $\nu_b$ is sufficiently small. Therefore,  $\Xi: \mathcal{T}^1_{\theta_b,C_0\nu_b}\cap \{v(0,t)=v_b(t)\} \rightarrow \mathcal{T}^1_{\theta_b,C_0\nu_b}\cap \{v(0,t)=v_b(t)\}$ is well-defined provided that $\nu_b$ is small enough. 
	
	
	For any $v, \vt \in \mathcal{T}^1_{\theta_b,C_0\nu_b}\cap \{v(0,t)=v_b(t)\}$, let $w=\Xi(v),\wt=\Xi(\vt)$. It yields that   
	\begin{align}
		&\sup_{x\in\R_-} \{e^{-\gamma x}\sum_{0\leq i+\frac{j}{2}\leq 1} \norm{\p_t^i\p_x^j (\hat{w}- \hat{\wt} -\hat{w}_{0,\infty}+\hat{\wt}_{0,\infty} ) }_{H^{1-i-\frac{j}{2}}_{\mathrm{per}}}\} +|\hat{w}_{0,\infty}-\hat{\wt}_{0,\infty}| 
		\notag\\
		\leq&  C  \sup_{x\in\R_-} \{e^{ -\gamma x} \norm{ f_1(v) \p_x v- f_1(\vt) \p_x \vt  }_{L^2_{\mathrm{per}}} \}  \notag\\
		\leq&  C  \sup_{x\in\R_-} \{e^{ -\gamma x} \big(\norm{ f_1(v)-f_1(\vt) }_{L^\infty_{\mathrm{per} } } \norm{\p_x v}_{L^2_{\mathrm{per}}} + \norm{ f_1(\vt) }_{L^{\infty}_{\mathrm{per} } } \norm{\p_x (v-\vt)}_{L^2_{\mathrm{per} } } \big) \}   \notag\\
		\leq&  C  \sup_{x\in\R_-} \{e^{ -\gamma x} \big(
		(  \sum_{0\leq i+\frac{j}{2}\leq 1} \norm{ (\hat{v}- \hat{\vt} -\hat{v}_{0,\infty}+\hat{\vt}_{0,\infty} ) }_{H^{1}_{\mathrm{per}}} +|\hat{v}_{0,\infty}-\hat{\vt}_{0,\infty}| )		
		 e^{\gamma x} M \notag\\
		 &\qquad\qquad+ M \norm{\p_x (v-\vt)}_{L^2_{\mathrm{per} } } \big) \}   \notag\\
		\leq&  C M \Big(  \sup_{x\in\R_-}\{ e^{-\gamma x}\sum_{0\leq i+\frac{j}{2}\leq 1} \norm{\p_t^i\p_x^j (\hat{v}- \hat{\vt} -\hat{v}_{0,\infty}+\hat{\vt}_{0,\infty} ) }_{H^{1-i-\frac{j}{2}}_{\mathrm{per}}} \} \\
		&\qquad\qquad+|\hat{v}_{0,\infty}-\hat{\vt}_{0,\infty}| \Big). \label{ineq-per-t1-3}
	\end{align}
	That is, $\Xi$ is a contraction mapping provided that $M=C_0\nu_b$ is small enough. 
	Consequently, by choosing $M=C_0\nu_b$ small enough in \eqref{ineq-per-t1m0} and \eqref{ineq-per-t1-3}, the existence and uniqueness of the strong solution $\ur$ in $\mathcal{T}^1_{\theta_b,C_0\nu_b}$ to \eqref{eqn-0} and \eqref{bc0} and its far field value $\ubr$ can be established by the Banach fixed point theorem. 
	The zero-mode formula \eqref{eqn-per-time-Fourier2} gives
	\begin{equation}
		|\ubr-\ub_b|
		=|\widehat v_{0,\infty}|
		\le C\sup_{x\in\mathbb R_-}
		\left\{e^{-\gamma x}
		\|f_1(v)\partial_xv\|_{L^2_{\rm per}}\right\}
		\le CM^2
		\le CC_0^2\nu_b^2.
	\end{equation}
	
	The remaining inequality, \eqref{ineq-per-t1}, can be derived
	by setting iterations as of $v^{(N+1)}=\Xi(v^{(N)})$ with $\hat{v}_k^{(1)}(x)= e^{r_+ x} \hat{v_b}_{,k}$ and $\hat{v}_{0,\infty}^{(1)}=0$, 
	and by letting $N\rightarrow\infty$ in similar inequalities as \eqref{ineq-per-t1m0}.
	
\end{proof}

\section{Stability of the inviscid shock}
The proof of \ref{thm-shock-viscous} is based on the analysis of the converging forward characteristics around the perturbed shock curve, and the delicate estimates on the shock curve. These two parts are given in the following subsections, respectively. 

\subsection{Behaviors of the forward characteristics and the solution}
In this subsection, the behaviors of the entropy solution are analyzed by the method of generalized characteristics.
We refer readers to \cite[Chapters 10 \& 11]{Dafer.2016} for details of properties of generalized characteristics, or \cite[Definition 3.1 \& Lemma 3.2]{Xin.Y.Y2019} for a short summary.

We denote that
\begin{equation}
	\label{def-forward-XGamma}
	\begin{aligned}
		\ull :&~\text{the entropy solution to the Cauchy problem \eqref{eqn-0} and \eqref{ic0},}\\
		&~~~ \text{~~for~}x\in \R, t\geq 0, ~\text{where~} u_0 \text{~is extended as constant~}\ubl \text{~for~} x>0,\\
		\ur :&~\text{the time-periodic solution to BVP \eqref{eqn-0} and \eqref{bc0} as above,}\\
		X_{1,2}(t) :&~ \text{the minimal and maximal, respectively, forward characteristic of~} u\\
		&~\text{~issuing from~} (0,0), \\
		\Gamma_1(t):&~ \text{one divide of~} \ull \text{~with value~} \ubl \text{~issuing from~} (x_0,0)~\text{with}~x_0\leq0,\\
		\Gamma_2(t):&~ \text{one divide of~} \ur \text{~with value~}  \ubr \text{~issuing from~} (0,t_b)~\text{with}~t_b\in[0,T).
	\end{aligned}
\end{equation}
Indeed, since $u_0\in L^1(\R)$, it follows from Lemma \ref{lem-div} that $x_0, t_b$ can be chosen as the ones satisfying 
$$\int_{0}^{x_0} (u_0-\ubl) dy =\min_{x\in \R} \int_{0}^{x} (u_0-\ubl) dy,$$ 
and \eqref{ineq-per-t0-2}, respectively.
In addition, $\Gamma_1'(t)=f'(\ubl)$, $\Gamma_2'(t)=f'(\ubr)$.
Note that \eqref{bc-incoming}, \eqref{ic-incominga} and \eqref{ic-incomingb} (or \eqref{ic-incoming1}) ensure that $X_{1,2}(t)$ will not intersect with the boundary; see the following lemma.

For conveniences, given $T$-periodic boundary data $u_b$ and with $\ubr$ defined in Lemma \ref{lem-per-t0}, we set
\begin{equation}
P(x):=\int_x^0\bigl(u_0(y)-\ubl\bigr)\, dy,\quad Q(t):=\int_0^t [f(u_b(\tau))-f(\ubr) ]\, d\tau.
	\label{eq:Q}
\end{equation}
$Q$ is also $T$-periodic and $Q(T)=0$. 
Therefore, \eqref{ic-incomingb} can be written in short as 
\begin{equation*}
	\label{ic-incomingb'} 
	\sup_{x\le X_I} P(x)<\min_{t\in[0,T]} Q(t).
\end{equation*}

\begin{lemma}\label{lem-hit}
	Under the assumptions of \eqref{thm-shock}, any forward characteristic of $u$ issuing from a point on the negative $x$-axis will not hit the boundary $t=0$. Moreover, in the case $s<0$, the same conclusion holds if, in the assumptions of \eqref{thm-shock}, \eqref{ic-incominga} and \eqref{ic-incomingb} are replaced by \eqref{ic-incoming1}.
\end{lemma}
\begin{proof}
    Suppose that the assumptions of \eqref{thm-shock} hold while Lemma \ref{lem-hit} does not hold, i.e. there exist a forward characteristic of $u$ issuing from a point on the negative $x$-axis and hitting the boundary for the first time at $t=T_B>0$. 
    Then denote $\eta(t)$ (for $0\le t\leq T_B$) the minimal backward characteristic of $u$ through $(0,T_B)$, and denote $v_B:=u(0-,T_B)$, $y_B:=\eta(0)$. Hence, it follows that  $\eta(t)$ is a straight line, and for $0<t<T_B$, 
    $$u(\eta(t)\pm,t)=u(0-,T_B)=v_B,\quad \eta'(t)=f'(v_B),\quad  y_B=-f'(v_B)T_B<0. $$
    Since the forward characteristic issuing from any point $(x,t)$ with $t>0$ is unique, and thus $\eta(t)$ lies on the left hand side of the forward characteristic; see Figure \ref{fig:hit}.
    \begin{figure}[H]
    	\centering
    	\begin{tikzpicture}[x=1cm,y=1cm,>=Latex]
    		\draw[->] (-5.8,0) -- (0.8,0) node[below] {$x$};
    		\draw[->] (0,-0.2) -- (0,5.1) node[right] {$t$};
    		\node[below right] at (0,0) {$0$};
    		
    		\coordinate (Y) at (-4.9,0);   
    		\coordinate (B) at (0,4.1);    
    		
    		\fill[black!6] (Y) -- (0,0) -- (B) -- cycle;
    		
    		\draw[dashed,thick] (0,0) -- (B) node[pos=0.56,right] {$x=0$};
    		\draw[dashed,thick] (Y) -- (B) node[pos=0.58,above left] {$\eta$};
    		
    		\draw[thick] (-2,0) .. controls (-1.15,1.15) and (-1.55,2.35) .. (-0.85,3.15)
    		.. controls (-0.45,3.65) and (-0.18,3.88) .. (B)
    		;
    		
    		\fill (Y) circle (1.2pt) node[below] {$(y_B,0)$};
    		\fill (B) circle (1.2pt) node[right] {$(0,T_B)$};
    	\end{tikzpicture}
    	\caption{Contradiction arguments for the forward characteristic hitting the boundary.
    	}
    	\label{fig:hit}
    \end{figure}
    Integrating \eqref{eqn-0} over the triangle with vertices $(y_B,0)$, $(0,0)$ and $(0,T_B)$ and applying Green's formula, gives the particularly simple identity
    \begin{equation}
    	P(y_B)-Q(T_B)+T_B [f(v_B)-f(\ubl)-f'(v_B)(v_B-\ubl)]=0.
    	\label{eq:triangle}
    \end{equation}
    Since $f$ is strictly convex, $f(v_B)-f(\ubl)-f'(v_B)(v_B-\ubl)\le 0$, and this reduces \eqref{eq:triangle} to 
    $$P(y_B)\geq Q(T_B).$$
    Noting that $f'(v_B)>0$, 
    the above inequality leads a contradiction to \eqref{ic-incomingb}, or just \eqref{ic-incoming1}). Therefore, Lemma \ref{lem-hit} is proved.
\end{proof}

Now we are ready to present the important proposition to characterize the behaviors of the entropy solution to IBVP \eqref{eqn-0}, \eqref{bc0}, and \eqref{ic0}.
\begin{proposition}\label{lem-gluing}
	Given the curves defined in \eqref{def-forward-XGamma}, for any $ t\geq 0, $ we further define two Lipschitz continuous curves:
	\begin{equation}\label{Def-curves}
		X_1^*(t):=\min\{X_1(t), \Gamma_1(t) \} \quad \text{ and } \quad X_2^*(t):= \max\{X_2(t), \Gamma_2(t)\}.
	\end{equation}
    Then the entropy solution $ u(x,t) $ to IBVP \eqref{eqn-0}, \eqref{bc0} and \eqref{ic0} satisfies
	\begin{equation}\label{eq-out}
		u(x,t) = \begin{cases}
		\ull(x,t) \quad \text{ if } x<X_1^*(t), \\
		\ur(x,t) \quad \text{ if } x>X_2^*(t).
		\end{cases}
	\end{equation}
\end{proposition} 

\begin{proof} 
Without loss of generality, we prove only the case when $x>X_2^*(t)$.

For any fixed $(\ovl{x},\ovl{t}) $ with $\ovl{x}>X_2^*(\ovl{t}),~ \ovl{t}>0$, we firstly prove that $u(\ovl{x}+,\ovl{t})=\ur(\ovl{x}+,\ovl{t})$.
Through $(\ovl{x},\ovl{t}) $ we draw the maximal backward characteristics $\xi_+(t)$ and $\eta_+(t)$ corresponding to the entropy solutions $u$ and $\ur$, respectively. It follows that $\xi_+(t)$ and $\eta_+(t)$ are both straight lines, and for $0<t<\ovl{t}$, 
$$ u(\xi_+(t)\pm,t)=u(\ovl{x}+,\ovl{t}),\quad \xi_+'(t)=f'(u(\ovl{x}+,\ovl{t})), $$
$$\ur(\eta_+(t)\pm,t)=\ur(\ovl{x}+,\ovl{t}),\quad \eta_+'(t)=f'(\ur(\ovl{x}+,\ovl{t})). $$
$\xi_+(t)$ (resp. $\eta_+(t)$) cannot cross through $X_2(t)$ (resp. $\Gamma_2(t)$) at $x<0$ since the forward characteristic issuing from any point $(x,t)$ with $t>0$ is unique. 
Therefore, $\xi_+(t)$ (resp. $\eta_+(t)$) lies above $X_2(t)$ (resp. $\Gamma_2(t)$), and hits the $t$-axis at $t_\xi$ (resp. $t_\eta$) above $0$.


If $u(\ovl{x}+,\ovl{t})> \ur(\ovl{x}+,\ovl{t})$, then $t_\xi<t_\eta$. By applying integration by parts to the equation of $u$ and $\ur$ on the triangle $(\ovl{x}, \ovl{t})$-$(0,t_\xi)$-$(0,t_\eta)$, one has
\begin{align}
	& \int_{t_\eta}^{\ovl{t}} \{ f(b)-f(u(\eta_+(t)-,t)) - f'(b)[b- u(\eta_+(t)-,t)] \} ~dt \notag \\
	& + \int_{t_\xi}^{\ovl{t}} \{ f(\widetilde{b})-f(\ur(\xi_+(t)+,t)) - f'(\widetilde{b})[\widetilde{b}-\ur(\xi_+(t)+,t)] \} ~dt \label{tril} \\
	& =\int_{t_\xi}^{t_\eta} 
	[f(u(0,t)) -f(\ur(0,t))] ~dt = 0.	\notag
\end{align}
here $b=\ur(\ovl{x}+,\ovl{t}), \widetilde{b}=u(\ovl{x}+,\ovl{t}),$ and the right hand side of \eqref{tril} equals to 0 since $u$ and $\ur$ satisfy the same boundary condition.
By the strict convexity of $f$, the left hand side of \eqref{tril} is non-positive, and thus for $t \in [0,\ovl{t}]$,
\begin{align}
	& u(\eta_+(t)-,t) \equiv b= \ur(\ovl{x}+,\ovl{t}), \quad \ur(\xi_+(t)+,t) \equiv \widetilde{b}=u(\ovl{x}+,\ovl{t}).\label{for1} 
\end{align}
Then \eqref{for1} implies that 
$$ f'(u(\eta_+(t)-,t))\equiv f'(\ur(\ovl{x}+,\ovl{t}))=\eta_+'(t),$$
which means that $\eta_+(t)$ is a backward generalized characteristic through $(\ovl{x},\ovl{t})$ associated with $u$. However, $\xi_+(t)$ is the maximal backward characteristic associated with $u$, thus there must hold $\eta_+(t) \leq \xi_+(t), t\in [0,\ovl{t}]$, which contradicts with $t_\xi<t_\eta$.

Similarly,  $u(\ovl{x}+,\ovl{t})< \ur(\ovl{x}+,\ovl{t}) $ could also lead a contradiction, and thus $u(\ovl{x}+,\ovl{t})= \ur(\ovl{x}+,\ovl{t}) $.
The proof to $ u(\ovl{x}-,\ovl{t})=\ur(\ovl{x}-,\ovl{t}) $ is similar by applying minimal backward characteristics.
So $u(x,t)=\ur(x,t)$ for $x>X_2^*(t)$.
\end{proof}

Note that $X_1^*$ and $X_2^*$ may not be the generalized characteristics of $u$. 
When $X_1^*$ and $X_2^*$ coincide after a large time as shown in Lemma \ref{lem-coincide}, they become the shock wave curve of $u$ after a sufficiently large time. The proof is motivated by \cite{Glimm.L1970,Liu.1977a} and our previous work \cite{Yuan.Y2020}.

\begin{lemma}\label{lem-coincide}
	There exists a finite time $ T_S>0, $ such that 
	$ X_1^*(t) = X_2^*(t) $ for all $ t\geq T_S. $
\end{lemma}

\begin{proof}
	We first define that for $u\geq v$, 
	\begin{equation}\label{Def-lam-sig}
		\varpi(u,v):= \int_0^1 f'(v+\theta (u-v)) d\theta.
	\end{equation}
	Therefore,  $\varpi(\ubl,\ubr)$ is the shock speed of the shock $ u^S$. Since $f$ is convex, it holds that
	\begin{align*}
		&	\varpi(u,v) \leq \varpi (u,s) \qquad \forall s\in[v,u],\\
        &	f'(v) =\varpi(v,v)\leq \varpi(u,v)\leq \varpi(u,u)=f'(u).
	\end{align*}

    Set $D(t):=X_2^*(t)-X_1^*(t)\geq 0$. It suffices to prove that there exists a finite time $ T_S>0, $ such that 
    $ D(t)\equiv 0 $ for all $ t\geq T_S. $ 
    
    It claims that there exists a finite time $ T_0>0 $ and $ \theta_0 \in (0,1), $ independent of time $ t, $ such that
    \begin{equation}\label{ineq-D}
    	D'(t) \leq ~\frac{\theta_0 D(t)}{t} + \frac{1-\theta_0}{2}\left( f'(\ubr) - f'(\ubl) \right), \quad t>T_0.
    \end{equation}
    Indeed, it can be established in the following cases.
	\begin{enumerate}
		\item[(i)] $ \Gamma_1(t) \leq X_1(t) \leq X_2(t) \leq  \Gamma_2(t). $ In this case, $D(t)=\Gamma_2(t)-\Gamma_1(t)$ and $D(t)'= f'(\ubr)-f'(\ubl) <0 $. Therefore, \eqref{ineq-D} holds.
		
		\item[(ii)] $ X_1(t)<\Gamma_1(t) < \Gamma_2(t) < X_2(t). $ In this case,  $D(t)=X_2(t)-X_1(t)$, then it follows from \Cref{lem-gluing} and mean-value theorem that for some $\theta(t)\in (0,1)$, 
		\begin{align*}
			D'(t) & = \varpi(u(X_2(t)-,t),u(X_2(t)+,t))- \varpi(u(X_1(t)-,t),u(X_1(t)+,t)) \\
			& = \varpi(u(X_2(t)-,t),\ur(X_2(t)+,t))- \varpi(\ull(X_1(t)-,t),u(X_1(t)+,t))\\
			& = \theta(t) [f'(u(X_2(t)-,t))-f'(u(X_1(t)+,t))] \\
			&\qquad+(1-\theta(t)) [f'(\ur(X_2(t)+,t))-f'(\ull(X_1(t)+,t))].
		\end{align*}
	   As stated in \cite{Liu.1977a}, $\theta(t)$ depends on $u(X_1(t)\pm, t), u(X_2(t)\pm, t)$, which range over a compact set in $u$-space, and therefore,  there exists a $ \theta_1 \in (0,1), $ independent of time $ t, $ such that 
	   \begin{align*}
	   	D'(t) &\leq \theta_1 [f'(u(X_2(t)-,t))-f'(u(X_1(t)+,t))] \\
	   	&\qquad+(1-\theta_1) [f'(\ur(X_2(t)+,t))-f'(\ull(X_1(t)+,t))].
	   \end{align*}
	   
       Since $X_{1,2}(t)$ are forward characteristic of $u$ issuing from the origin, the intersection of the maximal (resp. minimal) backward characteristic of $ u $ emanating from $ (X_1(t),t) $ (resp. $ (X_2(t),t) $) with $ x$-axis is the origin, which together with \eqref{ic-incoming1} gives that
       $$f'(u(X_2(t)-,t))=X_2(t)/t< 0, \quad f'(u(X_1(t)+,t))=X_1(t)/t<0.$$
       And thus, when $\Gamma_2(t)<X_2(t)$, there exists a $\delta_1\in(0,\delta_b$ such that
       \begin{equation}
       	\label{ineq-X2'}
       	X_2'(t)= \varpi(u(X_2(t)-,t),\ur(X_2(t)+,t)) <- \delta_1<0.
       \end{equation}
       With the aid of $\Gamma_2'(t)=f'(\ubr)<0$, it yields that \begin{equation}
       	X_2(t)\le-\delta_2t+C,
       	\quad \text{where~}
       	\delta_2:=\min\{-f'(\ubr),\delta_1\}>0,
       	\label{eq:X2linear}
       \end{equation} for large $t$.    
       It follows from the decay estimates of $u_\pm(x,t)$, Lemmas \ref{lem-per-t0} and \ref{lem-lr-decay} that 
		\begin{equation}
			\begin{aligned}
				D'(t)  &\leq \theta_1 \left( \frac{X_2(t)}{t} - \frac{X_1(t) }{t}\right) \\
				&\quad + \left( 1-\theta_1 \right) \left( f'(\ubr)+ \frac{C}{|X_2(t)|} - f'(\ubl) + \frac{C}{\sqrt{t}} \right), \\
				&\leq \theta_1 \frac{D(t)}{t} + \left( 1-\theta_1 \right) \left( f'(\ubr) - f'(\ubl) + \frac{C}{\sqrt{t}} \right).
			\end{aligned}
		\end{equation}
		So \eqref{ineq-D} holds true for a $ \theta_0>\theta_1 $ and a large enough $ T_0>0 $.
		
		\item[(iii)] $\Gamma_1(t) < X_1(t)< \Gamma_2(t) < X_2(t) $ or $ X_1(t)<\Gamma_1(t)  < X_2(t) < \Gamma_2(t). $ In these two cases, the proofs are same as the above case.
	\end{enumerate}

    Given the claim \eqref{ineq-D}, we have that
$$
\frac{d}{dt} \left(t^{-\theta_0}D(t)  \right) \leq \frac{1-\theta_0}{2} t^{-\theta_0} \left( f'(\ubr) - f'(\ubl) \right).
$$
Thus, for all $ t>T_0, $
\begin{equation}\label{ineq-D-1}
	D(t) \leq \frac{1}{2} \left( f'(\ubr) - f'(\ubl) \right) t + C(T_0) t^{\theta_0}.
\end{equation}
Since $ f'(\ubr) < f'(\ubl), $ \eqref{ineq-D-1} implies that $ D(t) = 0 $ after some positive time $ T_S>T_0 $. 

\end{proof}

\subsection{Estimates on the shock curve}

	 We define $X(t)=X_1^*(t)$. Hence, if follows from 
	 Lemma \ref{lem-coincide} that after $T_S$, $X(t)=X_1^*(t)=X_2^*(t)$. And together with \Cref{lem-gluing}, it holds that $u=\ull(x,t)$ if $x<X(t)$ and $u=\ur(x,t)$ if $x>X(t)$. 
	 In virtue of \Cref{lem-lr-decay}, 
	 \begin{equation}
	 	\label{decay-ulur}
	 	\ull=\ubl+\mathcal{O}(1)t^{-1/2},\quad \ur=\ubr+\mathcal{O}(1) x^{-1}.
	 \end{equation} 
	 Consequently, the decay estimates in \eqref{ineq-thm-shock1} is easily established. 

Therefore, to prove Theorem \ref{thm-shock}, it suffices to prove the estimates of $X(t)$, i.e., \eqref{ineq-shift} and \eqref{eq-finalshift} for $s<0$ and \eqref{eq:Xconv-per} for $s=0$ after $t>T_S$. Before the proof, we give a lemma about the integrals on $\ull$.
	 
	 \begin{lemma}
	 	\label{lem-bdd-divides}
	 	Assume that $u_0(x)-\ubl\in L^1(\R)$ for some constant $\ubl$. Let $\ull$ be the admissible solution to \eqref{eqn-0} and \eqref{ic0} for $\mu=0$.
	 	Then for any $\xb(t)\in\R$ satisfying $\xb(t)+\mathcal{O}(1)t^{1/2} \rightarrow -\infty$ as $t\rightarrow\infty$, define $\xi_1(\tau):=\xb(t)+ f'(\ubl)\tau$ for $0\leq \tau\leq t$, where $\mathcal{O}(1)$ is the same as that in  \eqref{decay-ulur}. It holds that $\text{as~}t\rightarrow\infty$, 
	 	\begin{equation*}
	 		\label{ineq-bdd-divide}
	 		\begin{aligned}
	 			\int_{0}^{t}  [f(\ull(\xi_1(\tau)\pm,\tau)) - f'(\ubl) \ull(\xi_1(\tau)\pm,\tau) ]  d\tau  
	 			- t[f(\ubl)&-f'(\ubl)\ubl] \rightarrow 0.
	 		\end{aligned}
	 	\end{equation*}
	 \end{lemma}

	 \begin{proof}
	 	Let $\eta_1(\tau)$ be the minimal backward characteristic, associated with $\ull$, that emanate from fixed point $(\xi_1(t),t)$. Thus, by the theory of the generalized characteristics of 1D convex scalar conservation law, $\eta_1(\tau)$ is a straight line and $\ull(\eta_1(\tau)\pm,\tau)\equiv \ull(\xi_1(t)-,t)$ for $\tau\in(0,t)$. ($\xi_1(\tau)$ is also the extremal backward characteristic associated with the constant solution $\ubl$.) 
	 	In virtue of \Cref{lem-lr-decay}, $\ull(\xi_1(t)-,t)=\ubl+\mathcal{O}(1)t^{-1/2}$, and thus $\xi_1(t)=\xb(t)+ f'(\ubl)t$, $\eta_1(0)=\xb(t)+\mathcal{O}(1)t^{1/2}$.
	 	
	 	Integrating \eqref{eqn-0} over the triangle with vertices $(\xi_1(t),t)$, $(\eta_1(0),0)$ and $(\xb,0)$, and applying Green's formula, it yields that 
	 	\begin{align*}
	 		&\int_{0}^{t} \{  f(\ubl)-f(\ull(\xi_1(\tau)\pm,\tau)) - f'(\ubl) [\ubl - \ull(\xi_1(\tau)\pm,\tau)] \} d\tau  \\
	 		+ &\int_{0}^{t} \{ f(\ull(\xi_1(t)-,t))- f(\ubl) - f'(\ull(\xi_1(t)-,t)) [\ull(\xi_1(t)-,t)-\ubl]  \} d\tau\\
	 		=	&-\int_{\eta_1(0)}^{\xb(t)} (u_0(y)-\ubl) dy
	 		\geq - \int_{-\infty}^{\xb(t)+\mathcal{O}(1)t^{1/2}} |u_0(y)-\ubl| dy
	 	\end{align*}
	 	Thanks to the convexity of $f$, both integrals on left hand side of the above inequality are non-positive. By the condition $\xb(t)+\mathcal{O}(1)t^{1/2} \rightarrow -\infty$ as $t\rightarrow\infty$ and $u_0(x)-\ubl\in L^1(\R)$, the right hand side of the above inequality tends to $0$ as $t\rightarrow\infty$. Therefore, \Cref{lem-bdd-divides} easily follows. 
	 \end{proof}

\begin{proof}[Proof of Theorem \ref{thm-shock}]
	As mentioned at the beginning of this subsection, it suffices to prove the estimates of $X(t)$, i.e., \eqref{ineq-shift} and \eqref{eq-finalshift} for $s<0$ and \eqref{eq:Xconv-per} for $s=0$ after $t>T_S$. The rest of the proof will be divided into two case.
	
	\vspace{3mm}
	
\noindent	(i) \textbf{Case $s<0$.} 	
	For fixed $t>T_S$ and $\tau\in[0,t]$, set
	 \begin{align*}
	 	\xi_1(\tau)&:=X(t)+f'(\ubl)(\tau-t) \quad \\
	 	\xi_2(\tau)&:\text{the divide of~} \ur \text{~with value~}  \ubr \text{~issuing from~} (0,mT+t_b) \\
	 	&~\text{for}~m\in\mathbb{N}_+\text{~s.t.~} \xi_2(t)-X(t)\in [0,f'(\ubr)T).
	 \end{align*}
	 So $\xi_1(\tau)$ emanates from $(X(t),t)$, hitting $t=0$ at $(X(t)-f'(\ubl)t,0)$, and $t_b$ satisfies \eqref{ineq-per-t0-2}.
	 
	 We first derive a rough estimate on $X(t)$, showing that $X(t)$ leaves away from the boundary. Since $u(x,t)=\ull(x,t)$ for $x<X(t)$ and $u(x,t)=\ur(x,t)$ for $X(t)<x<0$ after $t>T_S$, integrating the conservation law on $\mathbb{R}_-$ gives
	 \begin{equation*}
	 \begin{aligned}
	 	\int_{-\infty}^0 \bigl(u(x,t)-&\ubl\bigr)\,dx
	 	=
	 	\int_{-\infty}^0 \bigl(u_0(x)-\ubl\bigr)\,dx
	 	+\int_0^t \bigl[f(\ubl)-f(u_b(\tau))\bigr]\,d\tau  \\
	 	&=
	 	\int_{-\infty}^0 \bigl(u_0(x)-\ubl\bigr)\,dx
	 	+s(\ubl-\ubr)t
	 	-Q(t) .
	 \end{aligned}
	 \end{equation*}
	 On the other hand, by the above decomposition of $u$,
	 \[
	 \begin{aligned}
	 	\int_{-\infty}^0 \bigl(u(x,t)-\ubl\bigr)\,dx
	 	={}&
	 	\int_{-\infty}^{X(t)}
	 	\bigl(\ull(x,t)-\ubl\bigr)\,dx \\
	 	&+
	 	\int_{X(t)}^0
	 	\bigl(\ur(x,t)-\ubr\bigr)\,dx
	 	+(\ubl-\ubr)X(t).
	 \end{aligned}
	 \]
	 Consequently,
	 \begin{equation}
	 	\label{eq:integral-R}
	 	\begin{aligned}
	 		(\ubl-\ubr)\bigl(X(t)-st\bigr)
	 		={}&
	 		\int_{-\infty}^0 \bigl(u_0(x)-\ubl\bigr)\,dx
	 		-Q(t) \\
	 		&-
	 		\int_{-\infty}^{X(t)}
	 		\bigl(\ull(x,t)-\ubl\bigr)\,dx
	 		-
	 		\int_{X(t)}^0
	 		\bigl(\ur(x,t)-\ubr\bigr)\,dx .
	 	\end{aligned}
	 \end{equation}
	 Here $Q(t)$ is $T$-periodic and bounded, and
	 \[
	 \left|
	 \int_{-\infty}^{X(t)}
	 \bigl(\ull(x,t)-\ubl\bigr)\,dx
	 \right|
	 \le
	 \|\ull(\cdot,t)-\ubl\|_{L^1(\mathbb R)}
	 \le C, 
	 \]
	 and by $
	 |\ur(x,t)-\ubr|
	 \le \frac{C}{|x|} $ in Lemma~2.1,
	 \[
	 \begin{aligned}
	 	\left|
	 	\int_{X(t)}^0
	 	\bigl(\ur(x,t)-\ubr\bigr)\,dx
	 	\right|
	 	&\le
	 	C+
	 	C\int_{\min\{X(t),-1\}}^{-1}\frac{dx}{|x|} \\
	 	&\le
	 	C\bigl(1+\log(1+|X(t)|)\bigr).
	 \end{aligned}
	 \]
	 Therefore, it yields that
	 \begin{equation}
	 	\label{eq:X-rough}
	 	|X(t)-st|
	 	\le
	 	C\bigl(1+\log(1+|X(t)|)\bigr).
	 \end{equation}
	 We further have  
	 \[
	 |X(t)|
	 \le |s|t+C\bigl(1+\log(1+|X(t)|)\bigr)
	 \le |s|t+\frac12|X(t)|+C,
	 \]
	 which implies that $
	 |X(t)|\le C(1+t).$
	 Substituting this estimate back into \eqref{eq:X-rough} gives
	 \begin{equation}
	 	\label{eq:X-rough2}
	 	X(t)=st+O(\log t)
	 	\qquad\text{as }t\to\infty,
	 \end{equation}
	 
	 \eqref{eq:X-rough2} implies $X(t)<st/2$ for sufficiently large time, and thus it yields that $X'(t)=\jump{f(u)}/\jump{u}=s +\mathcal{O}(t^{-1/2})$. 
	 Then it is easy to find $T_1>T_S$ such that $\forall t>T_1$, $\xi_{i}(\tau)$  does not intersect with $\Gamma_{i}(t)$ for $i=1,2$. 
	 Integrating \eqref{eqn-0} over the pentagon with vertices $(X(t),t)$, $(\xi_1(0),0)$, $(0,0)$, $(0,\xi_2(0))$, and $(\xi_2(t),t)$ (see Figure \ref{fig:shift} below), and applying Green's theorem, Lemma \ref{lem-gluing}, it yields that
	 	\begin{align}
	 		&\int_{0}^{t}  [f(\ull(\xi_1(\tau)+,\tau)) - f'(\ubl) \ull(\xi_1(\tau)+,\tau) ]  d\tau
	 		+\int_{X(t)-f'(\ubl)t}^{0} u_0(y) dy \notag \\
	 		&\quad - \int_{0}^{mT+t_b} f(u_b(\tau)) d\tau
	 		-\int_{mT+t_b}^{t} [ f(\ur(\xi_2(\tau)-,\tau)) - f'(\ubr) \ur(\xi_2(\tau)-,\tau)]d\tau  \notag \\
	 		&\qquad  +\int_{\xi_2(t)}^{X(t)} \ur(y,t) dy=0. 	 	\label{eq:pentagon}
	 	\end{align}
	 \begin{figure}[htbp]
	 	\centering
	 	\begin{tikzpicture}[x=1cm,y=1cm,>=Latex, font=\small]
	 		\coordinate (A) at (-5.2,0);      
	 		\coordinate (O) at (0,0);         
	 		\coordinate (C) at (0,1.35);      
	 		\coordinate (D) at (-2.8,4.55);   
	 		\coordinate (E) at (-0.55,4.55);  
	 		\coordinate (M) at (-2.15,2.95);  
	 		
	 		\fill[black!6] (D) -- (A) -- (O) -- (C) -- (E) -- cycle;
	 		
	 		\draw[dashed,thick] (A) -- (D) node[pos=0.60,above left] {$\xi_1$};
	 		\draw[dashed,thick] (C) -- (E) node[pos=0.58,right] {$\xi_2$};
	 		\draw[dashed,thick] (O) -- (C) node[pos=0.56,right] {$x=0$};
	 		\draw[dashed,thick] (D) -- (E);
	 		
	 		\draw[thick] (O) .. controls (-1.0,0.9) and (-1.8,1.9) .. (M)
	 		node[pos=0.46,left=1pt] {$X_1$};
	 		
	 		\draw[thick] (O) .. controls (-0.2,1.0) and (-0.8,2.0) .. (M)
	 		node[pos=0.62,right=1pt] {$X_2$};
	 		
	 		\draw[thick] (M) .. controls (-2.7,3.7) and (-2.8,4.1) .. (D)
	 		node[pos=0.4,left=0pt] {$X$};
	 		
	 		\fill (A) circle (1.1pt) node[below] {$(\xi_1(0),0)$};
	 		\fill (C) circle (1.1pt) node[right] {$(0,\xi_2(0))$};
	 		\fill (D) circle (1.2pt) node[above left] {$(X(t),t)$};
	 		\fill (E) circle (1.2pt) node[above right] {$(\xi_2(t),t)$};
	 		
	 		\node at (-3.2,1.75) {$\Omega_t$};
	 		
	 		\draw[->] (-6.4,0) -- (0.9,0) node[below] {$x$};
	 		\draw[->] (0,-0.2) -- (0,5.5) node[right] {$t$};
	 		\node[below right] at (0,0) {$0$};
	 	\end{tikzpicture}
	 	\caption{Integration over the pentagon with vertices $(X(t),t)$, $(\xi_1(0),0)$, $(0,0)$, $(0,\xi_2(0))$, and $(\xi_2(t),t)$. }
	 	\label{fig:shift}
	 \end{figure}
	 In virtue of \Cref{lem-bdd-divides}, together with $X(t)-f'(\ubl) t +\mathcal{O}(1)t^{1/2}=st-f'(\ubl) t+\mathcal{O}(1)t^{1/2}\rightarrow -\infty$, and the decay estimates of $u_\pm$, it holds that as $t\rightarrow \infty$, 
	 \begin{align*}
	 	&\int_{0}^{t}  [f(\ull(\xi_1(\tau)+,\tau)) - f'(\ubl) \ull(\xi_1(\tau)+,\tau) ]  d\tau = t[f(\ubl)-f'(\ubl)\ubl] + o(1),\\
	 	&\int_{X(t)-f'(\ubl)t}^{0} u_0(y) dy = \ubl(-X(t)+f'(\ubl)t) + \int_{-\infty}^{0} (u_0(y)-\ubl) dy + o(1),\\
	 	&\int_{0}^{mT+t_b} f(u_b(\tau)) d\tau=\max\limits_{t\in[0,T]} Q(t) +f(\ubr)(mT+t_b),\\
	 	&\int_{mT+t_b}^{t} [ f(\ur(\xi_2(\tau)-,\tau)) - f'(\ubr) \ur(\xi_2(\tau)-,\tau)]d\tau=(t-mT-t_b)[f(\ubr)-f'(\ubr)\ubr],\\
	 	&\int_{\xi_2(t)}^{X(t)} \ur(y,t) dy = \ubr [X(t)-f'(\ubr)(t-mT-t_b)] + \mathcal{O}(|X(t)|^{-1}),
	 \end{align*}
	 which yields that 
	 \begin{equation}
	 	\label{ineq-bdd-per2}
	 	\jump{\ub}X(t)-\jump{f'(\ub)} t - \max\limits_{t\in[0,T]} Q(t)+ \int_{-\infty}^{0} (u_0(y)-\ubl) dy=  o(1).
	 \end{equation} 
	 Therefore, \eqref{ineq-shift} and \eqref{eq-finalshift} are proved.
	 	
	 \vspace{3mm}
	 
\noindent	(ii) \textbf{Case $s=0$. } Similar arguments as \eqref{eq:integral-R}, by integrating the conservation law on $\mathbb{R}_-$ and noting that $f(\ubl)=f(\ubr)$,  gives 
\begin{equation}
	\label{eq:integral-R'}
	\begin{aligned}
		\int_{X(t)}^0
		\bigl(\ubl - \ur(x,t)\bigr)\,dx
		={}&
		\int_{-\infty}^0 \bigl(u_0(x)-\ubl\bigr)\,dx
		-Q(t) \\
		&-
		\int_{-\infty}^{X(t)}
		\bigl(\ull(x,t)-\ubl\bigr)\,dx .
	\end{aligned}
\end{equation}
Since $f'(\ubr)<0<f'(\ubl)$, \eqref{bc-incoming} and $f$ is strictly convex, there exists a constant $c_0>0$ such that 
\begin{equation}
	\label{ineq-ur-ubl}
	\ubl-\ur(x,t)\ge c_0>0, \quad  \text{for}~x<0,~ t\ge0.
\end{equation}
Hence, \eqref{eq:integral-R'} implies that 
\[
|X(t)|\le C,\qquad t\ge T_S.
\]

Since $f'(\ubl)>0$, it follows that $X(t)-f'(\ubl)t+O(1)t^{1/2}\longrightarrow-\infty. $
Hence Lemma \ref{lem-bdd-divides} is applicable.
For fixed $t>T_S$, we choose $\xi_1$, $\xi_2$, $m$ and $t_b$ in the same way as in the case $s<0$, and apply Green's formula to the same pentagonal region, and the first four terms are treated exactly as above:
\begin{equation}
	\label{ineq-bdd-per2'}
	\begin{aligned}
		&-\ubl X(t)
		+\int_{-\infty}^0\bigl(u_0(y)-\ubl\bigr)\,dy
		-Q(t_b)\\
		&\quad
		+f'(\ubr)\ubr\,(t-mT-t_b)
		+\int_{\xi_2(t)}^{X(t)}\ur(y,t)\,dy
		=o(1),
	\end{aligned}
\end{equation}
where the only difference lies in the integral over the top segment between $(\xi_2(t),t)$ and $(X(t),t)$, since $X(t)$ remains bounded. On the other hand, integrating the equation for $\ur$ over the triangle region bounded by $x=0$ and the divide $\xi_2$, along which $\ur=\ubr$, yields
\begin{equation}
	\int_{\xi_2(t)}^0\ur(y,t)\,dy	=
	-f'(\ubr)\ubr\,(t-mT-t_b)-Q(t)+Q(t_b).
\end{equation}
Substituting above equality into \eqref{ineq-bdd-per2'}, we obtain
\begin{equation}
	\label{ineq-bdd-per3}
	\int_{X(t)}^0\bigl(\ubl-\ur(y,t)\bigr)\,dy 	=
	Q(t)-\int_{-\infty}^0\bigl(u_0(y)-\ubl\bigr)\,dy+o(1).
\end{equation}

We now define $X_{\rm per}(t)$ by
\begin{equation}
	\label{eq:Xper}
	\int_{X_{\rm per}(t)}^0
	\bigl(\ubl-\ur(y,t)\bigr)\,dy
	=
	Q(t)-\int_{-\infty}^0\bigl(u_0(y)-\ubl\bigr)\,dy .
\end{equation}
We verify that $X_{\rm per}$ is well-defined. Indeed, 
the assumption \eqref{ic-incomingb} implies
\[
Q(t)-\int_{-\infty}^0\bigl(u_0(y)-\ubl\bigr)\,dy>0
\qquad\text{for all }t.
\]
Moreover, by \eqref{ineq-ur-ubl}, for each fixed $t$, $
\Phi(x):=
\int_x^0\bigl(\ubl-\ur(y,t)\bigr)\,dy $
is continuous and strictly decreasing, equals zero at $x=0$, and tends to $+\infty$ as $x\to-\infty$. Thus \eqref{eq:Xper} uniquely determines $X_{\rm per}(t)<0$. Since both $\ur$ and $Q$ are $T$-periodic, uniqueness further implies $ X_{\rm per}(t+T)=X_{\rm per}(t). $

Subtracting \eqref{eq:Xper} from \eqref{ineq-bdd-per3} gives
\[
\int_{X(t)}^{X_{\rm per}(t)}
\bigl(\ubl-\ur(y,t)\bigr)\,dy=o(1),
\]
and thus by \eqref{ineq-ur-ubl} again, $ |X(t)-X_{\rm per}(t)|\rightarrow0$ as  $t\rightarrow\infty$. This completes the proof for $s=0$.

\end{proof} 

\begin{remark}
	\label{rem-Xper}
	The curve $X_{\rm per}(t)$ is precisely the shock curve connecting the left state $\ubl$ and the time-periodic right state $\ur$. Indeed, by differentiating 	\eqref{eq:Xper} with respect to $t$ and using
	\[ 	\partial_t\ur+\partial_x f(\ur)=0,	\quad 
	Q'(t)=f(u_b(t))-f(\ubr), \quad
	f(\ubl)=f(\ubr),	\]
	we have
	\begin{equation}
		\label{eq:Xper2}
		X_{\rm per}'(t)
		=\frac{ f(\ubl)-f(\ur(X_{\rm per}(t)+,t)) }{ 	\ubl-\ur(X_{\rm per}(t)+,t) },
	\end{equation}
	which is exactly the Rankine--Hugoniot condition. Since $f$ is strictly convex and $\ubl>\ur(X_{\rm per}(t)+,t)$, this discontinuity is an entropy shock. Hence $X_{\rm per}$ describes a $T$-periodic shock connecting $\ubl$ and $\ur$.
\end{remark}

\section{Stability of the viscous shock}

The proof of Theorem \ref{thm-shock-viscous} is based on the weighted energy estimates developed by \cite{Liu.N1997}.
Using the carefully constructed ansatz $\usharp$ in \eqref{def-ansatz-shock}, we reformulate the original problem \eqref{eqn-0}, \eqref{bc0} and \eqref{ic0} as a problem for the anti-derivative of $u-\usharp$.

\subsection{Reformulation of the problem}
Since $\ur$ and $\phi^S$ are solutions to \eqref{eqn-0}, $\usharp$ defined in \eqref{def-ansatz-shock} satisfies
\begin{equation}
	\label{eqn-usha-1}
	\p_t\usharp+ \p_x f(\usharp) -\p_x^2\usharp=h,
\end{equation}
where by denoting $\ull=\ubl$, the source term $h$ is 
\begin{align}
	h=&  -\jump{u}\p_x\sigma_{\X}(s+ \X') + \p_x f(\usharp) -\p_x f(\ur) \sigma_{\X} \notag\\
	&\quad -\jump{u}\p_x^2\sigma_{\X}-2\jump{\p_x u} \p_x\sigma_{\X}  \notag\\
	=&\left\{ - \jump{u}(s+ \X') + \jump{f(u)-\p_x u} \right\} \p_x\sigma_{\X}  	\notag\\
	&\quad +\p_x \left\{ f(\usharp)-f(\ubl)(1-\sigma_{\X})-f(\ur)\sigma_{\X}-\jump{u} \p_x \sigma_{\X}  \right\}. 		 \label{eqn-source-1}	
\end{align} 
In virtue of the equation of $\phi^S$ \eqref{eqn-shock1} and the Rankine-Hugoniot condition \eqref{RH}, $h$ can be rewritten as
\begin{equation}
	\label{eqn-source-2}
	\begin{aligned}
		h&=-\X' \p_x\phi^S_{\X}+\p_x H_1+ h_2 \,,  \\ 
		\text{where~}H_1&=	 f(\usharp)-f(\phi^S_{\X}) -(f(\ur)-f(\ubr))\sigma_{\X}- (\ur - \ubr) \p_x \sigma_{\X} \,,  	\\
		h_2&=\left\{ - (\ur-\ubr)(s+\X') + (f(\ur)-f(\ubr)-\p_x \ur) \right\} \p_x\sigma_{\X}
		\,, \\
		\text{and~} H_2 &=\int_{-\infty}^{x} h_2(y,t) dy\,.
	\end{aligned}	
\end{equation}

Note that when $\ur(t)\equiv\ubr$, $\usharp\equiv\phi^S_{\X}$ and $H_1=h_2\equiv 0$. Therefore, $H_1, h_2$ are source terms induced by the time-periodic boundary. In the following, we will prove that $H_1$ and $h_2$ are well decaying with respect to both $x$ and $t$. 

To enable that $U(x,t)$ is well-defined, the shift curves $\X$ are set to satisfy 
\begin{equation}
	\label{cond-shift}
	\int_{-\infty}^0 (u-\usharp) (x,t) dx =0 \quad \text{for all~}t\geq0.
\end{equation}
The above equation hold when it hold initially and $\frac{d}{dt}\int_{-\infty}^0 (u-\usharp) (x,t) dx =0$  for all $t>0$.
It follows from \eqref{eqn-usha-1} and \eqref{eqn-source-2} that $\X(t)$ has to satisfy
\begin{equation}
	\label{eqn-shift1}
	\begin{aligned}
		\Big(\int_{-\infty}^0 \jump{u} \p_x \sigma_{\X} dx\Big)&  \X'(t) = \Big( f(\ubr)-f(\phi^S_{\X})\big|_{x=0}-\p_x^2 U(0,t) \Big)   \\
		& +\int_{-\infty}^0 \{ - (\ur - \ubr) s+ f(\ur)-f(\ubr)-\p_x \ur \} \p_x\sigma_{\X} dx\\
		& +\Big( (f(u_b)-f(\ubr))(1-\sigma_{\X}) - (u_b-\ubr) \p_x\sigma_{\X}  \Big)\big|_{x=0} \,	,
	\end{aligned}
\end{equation}
and $\X_0 :=\X(0)$ is determined by
\begin{equation}
	\label{def-shift1}
	\int_{-\infty}^0 \{u_0(x)- \phi^S(x-\X_0) - (\ur(x,0)-\ubr) \sigma(x-\X_0) \}\, dx =0 \,.
\end{equation}
It is noted that the first line of the right hand side of \eqref{eqn-shift1} is the similar as \cite{Liu.N1997}, the reminders are quadratic terms and decay exponentially to $0$ when $t\rightarrow+\infty$; see Lemma \ref{lem-shift1}.

Thanks to \eqref{eqn-usha-1} and \eqref{cond-shift}, the anti-derivative $U(x,t)$ satisfies
\begin{equation}
	\label{eqn-anti-deriv}
	\begin{aligned}
		\p_t U+\X'(\ubl-\phi^S_{\X})+ &f(\usharp+\p_x U)-f(\usharp) -\p_x^2 U=-H_1-H_2, 	\\
		U(x,0) & = U_0(x), \quad\qquad x\in \R_- ,		
		\\
		U(-\infty,t)&=U(0,t) =0, \quad   t\geq 0, 
	\end{aligned} 
\end{equation}
where $H_1, H_2$  are defined in \eqref{eqn-source-2}. 

\vline

The theorem for the stability of the viscous shock, Theorem \ref{thm-shock-viscous}, follows from the following theorem on the anti-derivative $U$.

The solution spaces $\mathcal{S}(0,T)$ and $\mathcal{Z}(0,T)$ for $U$ and $\X$ are, respectively,
\begin{align*}
	\mathcal{S}(0,T)=& \big\{ U\in C([0,T]; H_{\beta}^3), ~ \p_x U\in L^2(0,T; H_{\beta}^3), ~\text{and}\\
	&\quad(1+\X_0 +t)^{\beta+3}\abs{\p_x^2 U(0,t) }^2+ \int_{0}^{t}(1+\X_0 +\tau)^{\beta+3}\abs{\p_x^2 U(0,\tau)}^2<\infty \big\},\\
	\mathcal{Z}(0,T)=&\big\{\X(t)\in C^1[0,T],~ \abs{\X(t)-\X_0}<\infty \big\},
\end{align*}
where $\beta>0$.
For simplicity, we define 
\begin{equation}
\label{small-periodic}
 \nu:=\norm{u_b(t)-\ub_b}_{H^3_{\mathrm{per}}(0,T)}.
\end{equation}  
\begin{theorem}
	\label{thm-anti-deriv}
	Assume $\mu=1$, and that $ u_b\in$ $ H^1_{\mathrm{per}}(0,T) $ is periodic with period $T$ and average $ \ub_b$, and that $\ur(x,t)$ is the time-periodic entropy solution in Lemma \ref{lem-per-t1} with $\ur(x,t)\rightarrow \ubr $ as 	$x\rightarrow-\infty$. Furthermore, assume \eqref{RH} with $s<0$, \eqref{E}, and that the shock is non-degenerate, i.e. \eqref{non-degenerate-shock}
	. 
	
	Then there exists $ \varepsilon_0>0$ such that if
	$$  \norm{U_0}_{H_{\beta}^3}+|\X_0|^{-1} + \nu < \varepsilon_0, $$ 
	then the problem \eqref{eqn-anti-deriv} and \eqref{eqn-shift1} admits a unique global solution $ (U,\X) \in \mathcal{S}(0,\infty) \times \mathcal{Z}(0,\infty)$ satisfying
	\begin{equation*}
		\label{def-final shift}
		\X'\in L^1(0,\infty) ~ \text{and for some constant } \X_\infty,~
		 \X(t)\rightarrow  \X_\infty \text{ as } t\rightarrow \infty.
	\end{equation*}
\end{theorem}

\subsection{A priori estimates}
The global existence theorem will be obtained by combining the local existence and a priori estimates. The local-in-time existence of the classical solution to \eqref{eqn-anti-deriv} and \eqref{eqn-shift1} with constant boundary value has been treated in \cite{Liu.N1997} by iteration arguments. The time-periodic boundary values do not make a difference for the local-in-time existence.
So it suffices to prove the a priori estimates as follows.

 Set the a priori assumption
\begin{equation}
	\label{eqn-assum-shock}
	\begin{aligned}
		\varepsilon:=&\sup\limits_{t\in[0,T]} \Big\{ \norm{U(t)}_{H^3_{\beta}}^2 + (1+|\X_0|+t)^{\beta+3} |\p_x^2 U(0,t)|^2 \\
		&\qquad +\int_{0}^{t} (1+|\X_0|+\tau)^{\beta+3} |\p_x^2 U(0,\tau )|^2  d \tau  \Big\} ,
	\end{aligned}
\end{equation}

\begin{proposition}[A priori estimates]
	\label{prop-apriori-shock}
	Under the assumptions of Theorem \ref{thm-anti-deriv}, for any $T>0$, assume that $ (U,\X) \in \mathcal{S}(0,T) \times \mathcal{Z}(0,T)$ is the solution to \eqref{eqn-anti-deriv} and \eqref{eqn-shift1}. Then there exist positive constants $\nu_0, \varepsilon_0$ independent of $T$ such that if $\nu<\nu_0$, and $\varepsilon+|\X_0|^{-1}\leq \varepsilon_0$, 
	then 
	\begin{equation}
		\label{ineq-apriori-shock}
		\begin{aligned}
			\sup_{t\in[0,T]} &\Big\{\norm{U(t)}_{H^3_{\beta}}^2
			 + (1+|\X_0|+t)^{\beta+3} |\p_x^2 U(0,t)|^2 \Big\} + \int_{0}^{T}  \Big\{ \norm{\p_x U(t)}_{H^3_{\beta}}^2\\
			&+ (1+|\X_0|+t)^{\beta+3} |\p_x^2 U(0,t )|^2  \Big\} d t \leq C \Big( \norm{U_0}_{H^3_{\beta}}^2 +|\X_0|^{-1}+ \nu\Big).
		\end{aligned}
	\end{equation}
\end{proposition}

The proof of Proposition \ref{prop-apriori-shock} consists of the following series of lemmas.

\subsubsection{The shift and source terms}

\begin{lemma}
	\label{lem-shift1}
	Under the assumptions of \Cref{prop-apriori-shock}, there exists small positive constant $\nu_1,\varepsilon_1$ independent of $T$ such that, if $\nu<\nu_1$ and $\varepsilon<\varepsilon_1$, then $\text{for~} 0\leq t\leq T,$ we have 
	\begin{equation}\label{ineq-shift1}
		\begin{aligned}
			st+\X(t)&\leq \X_0, \quad |\X(t)-\X_0|\leq C, \\
			|\X'(t)| &\leq C e^{-c(|\X_0|+t)}+C\abs{\p_x^2 U(0,t)},\\
			\int_{0}^{t} (1+|\X_0|+\tau)^{\beta+3} |\X'(\tau)|^2 d\tau &\leq Ce^{-c|\X_0|}\\
			&\quad+C\int_{0}^{t} (1+|\X_0|+\tau)^{\beta+3} \abs{\p_x^2 U(0,t)}^2 d\tau.
		\end{aligned} 
	\end{equation}
\end{lemma}
\begin{proof} 
 (1) We first concern the upper bound of the shift $st+\X(t)$.
	It follows from estimates of the time-periodic solution $\ur$ in  \eqref{ineq-per-t1} for $m=3$
	, shock profile $\phi^S$ \eqref{ineq-shock1} and the a priori assumption \eqref{eqn-assum-shock} that if $\nu<\nu_b$ is small enough, then 
	\begin{align*}
		&\int_{-\infty}^{0} \jump{u} \p_x \sigma_{\X} dx  \leq  \int_{-\infty}^{0} \frac{1}{2} \jump{\ub} \p_x \sigma_{\X}  dx = \frac{1}{2} \big(\phi^S(-st-\X(t))-\ubl\big)<0, \\
		& \abs{-\p_x^2 U(0,t)}  
		 +\abs{\int_{-\infty}^0 \{ - (\ur - \ubr) s+ f(\ur)-f(\ubr)-\p_x \ur \} \p_x\sigma_{\X} dx}\\
		& \qquad\qquad +\abs{\Big( (f(u_b)-f(\ubr))(1-\sigma_{\X}) - (u_b-\ubr) \p_x\sigma_{\X}  \Big)\big|_{x=0} } \leq \varepsilon+C\nu.
	\end{align*}
In virtue of \eqref{E} and $s<0$, we take $\nu_1, \varepsilon_1$ such that if $\nu<\nu_1$ and $\varepsilon<\varepsilon_1$, then
\begin{align*}
	f(\ubr)-f(\phi^S(-\X_0))>\varepsilon+C\nu,
\end{align*}
and therefore by \eqref{eqn-shift1}, $st+\X(t)\leq \X_0$ for all $t\in[0,T]$.

Otherwise, by contradiction, if $t_0:=\inf\{t\in(0,T]; st+\X(t)> \X_0 \}$ exists, then it follows from the continuity assumption on $\X$ in \Cref{prop-apriori-shock}  that $st_0+\X(t_0) = \X_0$, and thus by \eqref{eqn-shift1},
\begin{align*}
	\int_{-\infty}^{0} \jump{u} \p_x \sigma_{\X} dx \Big\vert_{t=t_0} \X'(t_0)
	&> f(\ubr)-f(\phi^S(-\X_0))-\varepsilon-C\nu>0.
\end{align*}
Hence, $\X'(t_0)<0$, and $st+\X(t)<\X_0$ in a small neighbourhood of $t_0$, which contradicts to the definition of $t_0$.

\vspace{0.2cm}

(2) Now we estimate $\X'(t)$ and $\X(t)$. The upper bound of $st+\X(t)$ gives that 
\begin{align*}
	&0<\ubl-\phi^S(-\X_0)<\ubl-\phi^S(-st-\X) \leq \ubl-\ubr,\\
	&f(\ubr)-f(\phi^S(-st-\X)) \sim \abs{\phi^S(-st-\X)- \ubr} \sim e^{\theta_s(st+\X)}.
\end{align*}
As a result, 
\begin{align*}
	&\int_{-\infty}^{0} \jump{u} \p_x \sigma_{\X} dx  \leq  \frac{1}{2} \big(\phi^S(-\X_0)-\ubl\big)<0,\\
	&\abs{\int_{-\infty}^0 \{ -(\ur-\ubr) s+ f(\ur)-f(\ubr)-\p_x\ur \} \p_x\sigma_{\X} dx}\\
	&\quad \leq C\int_{-\infty}^0 \nu e^{\theta_b x} e^{-\theta_s |x-st-\X| } dx\\
	&\quad \leq C \nu \int_{-\infty}^{\frac{st+\X}{2}} e^{\theta_b x} dx+ C\nu\int_{\frac{st+\X}{2}}^0 e^{-\theta_s |x-st-\X| } dx\\
	&\quad \leq C \nu e^{\min\{\theta_b,\theta_s\} (st+\X)/2},\\
	&\abs{\big( (f(u_b)-f(\ubr))(1-\sigma_{\X}) - (u_b-\ubr) \p_x\sigma_{\X}  \big)\big|_{x=0} } \leq C \nu e^{\theta_s(st+\X)}.
\end{align*}
By the equation of $\X'(t)$, \eqref{eqn-shift1}, it yields that
\begin{equation} \label{ineq-shift11}
   c e^{c(st+\X)} -C|\p_x^2 U(0,t)|	\leq \X'(t)\leq C e^{c(st+\X)} +C|\p_x^2 U(0,t)|.
\end{equation}
Therefore, 
$\X(t)-\X_0 \geq -C\int_{0}^{t} |\p_x^2 U(0,t)| d\tau \geq -C \varepsilon$,
, and \subeqref{ineq-shift1}{2}, \subeqref{ineq-shift1}{1} and \subeqref{ineq-shift1}{3} can be successively obtained. 
\end{proof}

Thanks to the exponentially decaying of the shock profile \eqref{ineq-shock1} and the periodic solution \eqref{ineq-per-t1}, the quadratic source terms $H_1$ and $h_2$ are well decaying both in $t$ and $x$: 
\begin{lemma}
	\label{lem-source}
	Under the assumptions of \Cref{prop-apriori-shock}, $H_1$ and $h_2$ satisfy that
	\begin{equation}\label{ineq-H}
		\norm{ H_1 }_{H^3} +\norm{ h_2 }_{H^2} + \norm{ H_1 }_{H_{\beta}^3} + \norm{ h_2 }_{H_{\beta+1}^2} \leq  C\nu e^{-c(|\X_0|+t)}.
	\end{equation}
\end{lemma}
The proof is based on Lemmas \ref{lem-per-t1} and \ref{lem-shift1}.
\begin{proof}
	For $H_1$, it follows from \eqref{def-ansatz-shock} and \eqref{eqn-source-2} that 
	\begin{align*}
		H_1=&\int_{0}^{1} f'(\phi^S_\X+\rho_1(\ur-\ubr) \sigma_{\X}) d\rho_1  (\ur-\ubr) \sigma_{\X} \\
		&\quad + \int_{0}^{1} f'(\ubr+\rho_1(\ur-\ubr) ) d\rho_1  (\ur-\ubr) \sigma_{\X}  - (\ur - \ubr) \p_x \sigma_{\X},\\
		=& -\int_{0}^{1}\int_{0}^{1}  f''(\ubr+\rho_1(\ur-\ubr) - \rho_2 (\ubr-\ubl+\rho_1(\ur-\ubr) )(1-\sigma_{\X}) )\\
		&\quad \cdot (\ubr-\ubl+\rho_1(\ur-\ubr) )(1-\sigma_{\X}) d\rho_2 d\rho_1 \cdot (\ur-\ubr)\sigma_{\X} \\
		&\quad   - (\ur - \ubr) \p_x \sigma_{\X} \, ,\\
	   |H_1|
		\leq  & C \abs{\ur-\ubr}((1-\sigma_{\X})\sigma_{\X}  + \abs{\p_x\sigma_{\X}}).
	\end{align*}
	It follows from the decaying properties of the shock profile \eqref{ineq-shock1}, time-periodic solution \eqref{ineq-per-t1} for $m=3$, and the shift \eqref{ineq-shift1} that for $k=0,1,2,3$, 
	\begin{align}
		\norm{\p_x^k H_1}^2\leq& C  \int_{-\infty}^0 \nu e^{2\theta_b x} e^{-2\theta_s|x-st-\X|} dx     \notag\\
		\leq & C\nu \big(e^{\theta_b (st+\X) }\int_{-\infty}^{\frac{st+\X}{2}}  e^{-2\theta_s|x-st-\X|} dx
		      +e^{\theta_s(st+\X)}\int_{\frac{st+\X}{2}} ^0 e^{2\theta_b x}  dx \big)    \notag\\
		\leq &  C\nu e^{\min\{\theta_b,\theta_s\} (st+\X)} \leq C\nu e^{-c(|\X_0|+t)} . \label{ineq-H-1}
	\end{align}
   
   It follows from \eqref{eqn-assum-shock} and \eqref{ineq-shift1} that $|\X'|\leq C$, and then the estimates of $h_2$ can be also obtained by separating the integrated interval into $(-\infty,(st+\X)/2]$ and $ [(st+\X)/2,0]$: for $k=0,1,2$,
   \begin{align}
   	\norm{\p_x^k h_2}^2\leq&   
    C\nu \int_{-\infty}^0 e^{2\theta_b x} e^{-2\theta_s|x-st-\X|} dx \leq C\nu e^{-c(|\X_0|+t)}.     \label{ineq-H-2}
   \end{align}
   The estimates of $\abs{ \p_x^k H_1 }_{\beta+k}$ and $\abs{ \p_x^k h_2 }_{\beta+k+1}$ are similar.
   So the proof of lemma is completed.

\end{proof}

\subsubsection{Basic a priori estimate.}

Linearizing \eqref{eqn-anti-deriv} with respect to $\usharp$ gives that 
\begin{equation}
		\p_t U+ f'(\phi^S_\X) \p_x U-\p_x^2 U=(f'(\phi^S_\X)-f'(\usharp)) \p_x U-\X'(\ubl-\phi^S_{\X})+ H_0-H_1-H_2\,, 
		\label{eqn-anti-deriv1}
\end{equation}
where \begin{equation}
	H_0=-f(\usharp+\p_x U) + f(\usharp)+ f'(\usharp)\p_x U =\mathcal{O} (|\p_x U|^2). \label{def-H0}
\end{equation}

Compared with constant boundary effect in \cite{Liu.N1997}, the new difficulties to perform energy estimates on \eqref{eqn-anti-deriv} lay in the new terms, $(f'(\phi^S_{\X}) -f'(\usharp)) \p_x U, H_1, H_2$, which are generated by the time-periodic boundary.
To make this proof complete and clear, we still present the details of all the estimates here.

Thanks to our careful designed ansatz, these new terms 
are quadratic, which consists of time-periodic parts (decaying away from the boundary) and shock profile parts (decaying away from the shift $st+\X$). So these new terms are considered in separated integrated intervals $(-\infty,st+\X), (st+\X, (st+\X)/2)$ and $((st+\X)/2,0]$.

\begin{lemma}
	\label{lem-eng-basic}
	Under the assumptions of \Cref{prop-apriori-shock}, there exists small positive constant $\nu_2,\varepsilon_2$ independent of $T$ such that, if $\nu<\nu_2$ and $\varepsilon+|\X_0|^{-1}<\varepsilon_2$, then we have
	\begin{equation}\label{eng-basic}
		\begin{aligned}
			\sup_{t\in[0,T]} 
			&   \norm{U(t)}^2  +\int_{0}^{T}\big(\|\sqrt{-\p_x\phi^S_{\X}} U \| ^2+ \norm{\p_x U}^2\big)   d t \\
			&\leq C (\norm{U_0}^2 + |\X_0|^{-1})  + C \X_0^{-\beta} \int_{0}^{T} (1+|\X_0|+t)^{\beta+3} |\p_x^2 U(0,t )|^2  d t + C\nu.
		\end{aligned} 
	\end{equation}
\end{lemma}

\begin{proof}
    As \cite{Liu.N1997}, we define the weight function 
	\begin{equation*}
		\label{def-weight1}
		w(\phi)=\frac{(\phi-\ubl)(\phi-\ubr)}{f_0(\phi)},
	\end{equation*}
    which satisfies that
    \begin{equation}
    	\label{prop-weight}
    	\begin{aligned}
    		&(w f_0)'(\phi)=2\phi-(\ubr+\ubl),\quad (w f_0)''(\phi)=2>0,\\
    		&\qquad
    		c\leq w(\phi)\leq C, \quad \abs{w'(\phi)}\leq C.
    	\end{aligned}
    \end{equation}

   By multiplying \eqref{eqn-anti-deriv1} with $w(\phi^S_{\X}) U$, integrating the resultant on $(-\infty,0)$, using integration by parts and noting that $U(0,t)=0$, we have
   \begin{align}
     &\dt \int_{-\infty}^{0} \frac{1}{2} w(\phi^S_{\X}) U^2 dx +  \int_{-\infty}^{0} \{\frac{1}{2} [(w f_0)''(\phi^S_{\X}) - \X' w'(\phi^S_{\X})] (-\p_x \phi^S_{\X}) U^2 +w(\phi^S_{\X})|\p_x U|^2\} dx \notag \\
     &\ = \int_{-\infty}^{0}  w(\phi^S_{\X}) \{ (f'(\phi^S_{\X}) -f'(\usharp)) \p_x U- \X' (\ubl-\phi^S_{\X}) + H_0-H_1-H_2 \} U  dx, \label{eng-basic-1}
   \end{align}
  where in virtue of \eqref{prop-weight} and \eqref{eqn-assum-shock},
  \begin{equation}
  	[(w f_0)''(\phi^S_{\X}) - \X' w'(\phi^S_{\X})] \geq 2-C\X' \geq 2-C\varepsilon>1 \qquad\text{for small~}\varepsilon. \label{eng-basic-1'}
  \end{equation}
  It remains to estimate the right hand terms of \eqref{eng-basic-1}.

   It follows from $f\in C^4$, the definition of the ansatz \eqref{def-ansatz-shock}, the H\"{o}lder inequality and Hardy inequality that
   \begin{align}
     &\abs{\int_{-\infty}^{0} w(\phi^S_{\X}) (f'(\phi^S_{\X}) -f'(\usharp)) \p_x U U  dx} \leq C \int_{-\infty}^{0} \abs{\ur-\ubr} \abs{\sigma_{\X}} |\p_x U| |U| dx \notag\\
     &\qquad\qquad
     \leq C \norm{x(\ur-\ubr)}_{L^{\infty}} \norm{\sigma_{\X}}_{L^{\infty}} \norm{x^{-1}U} \norm{\p_x U} 
     \leq  C\nu \norm{\p_x U}^2.      \label{eng-basic-2}
   \end{align}
  Moreover, with \eqref{eqn-assum-shock}, \eqref{def-H0}, \eqref{ineq-H} and the Young inequality, it yields that
   \begin{align}
   	&\abs{\int_{-\infty}^{0} w(\phi^S_{\X}) H_0 U  dx} \leq C \norm{U}_{L^\infty} \norm{\p_x U}^2 \leq  C\varepsilon \norm{\p_x U}^2, \label{eng-basic-3} \\
   	&\abs{\int_{-\infty}^{0} w(\phi^S_{\X}) H_1 U  dx} \leq C \norm{H_1} \norm{U} 
   	\leq C\nu e^{-c(|\X_0|+t)} (1+\norm{U}^2). \label{eng-basic-4}
   \end{align}
   The rest terms in the right hand side of \eqref{eng-basic} are 
   \begin{align}
   	&\abs{\int_{-\infty}^{0}  w(\phi^S_{\X})  (-\X' (\ubl-\phi^S_{\X}) -H_2 ) U  dx} \notag\\
   	&\qquad\leq C \int_{-\infty}^{0} (| \X' | |\ubl-\phi^S_{\X}| +  |H_2|) |U|  dx 
   	=C \left(\int_{-\infty}^{st+\X} +  \int_{st+\X}^0\right) \,. \label{eng-basic-50}
   \end{align}
   (i) For $x\leq st+\X$, noting that $\p_x \sigma>0$ for $\xi\in \R$ and 
   $|\ubl-\phi^S(\xi)|\leq C|\p_x \phi^S(\xi)|$ for $\xi\leq 0$, 
   it yields that 
   \begin{align*}
   	\abs{H_2(x,t)}
   	&\leq C   \nu e^{\theta_b x}  \left(\int_{-\infty}^{x} \p_x \sigma_{\X} dy \right) \leq C  \nu e^{\theta_b (st+\X)/2}  |\sigma_{\X}| \leq C  \nu e^{-c(|\X_0|+t)} |\p_x \phi^S_{\X}|  ,
   \end{align*}
   and thus 
   \begin{align}
   	&\int_{-\infty}^{st+\X} (| \X' | |\ubl-\phi^S_{\X}| +  |H_2|) |U|  dx 
   	\leq 
   	C (| \X' |+\nu e^{-c(|\X_0|+t)})\int_{-\infty}^{st+\X} |\p_x \phi^S_{\X}| |U|  dx 
   	.\label{eng-basic-51}
   \end{align}
   (ii) For $st+\X<x\leq 0$, it follows from \eqref{def-shift1} that, similar as the estimates of $\X'(t)$ in Lemma \ref{lem-shift1}, 
   \begin{align*}
   	-H_2(0,t)
   	&=\X' (\ubl-\phi^S_{\X})+ \Big( f(\ubr)-f(\phi^S_{\X})\big|_{x=0}-\p_x^2 U(0,t) \Big)   \\
   	& \qquad +\Big( (f(u_b)-f(\ubr))(1-\sigma_{\X}) - (u_b-\ubr) \p_x\sigma_{\X}  \Big)\big|_{x=0} 	\\
   	&= \mathcal{O} ( |\X'| +e^{-c(|\X_0|+t)}+\abs{\p_x^2 U(0,t)}  )\,,   	
   \end{align*}
   Making use of $|\X(t)-\X_0|\leq C$ in \eqref{ineq-shift1}, we have $|st+\X(t)|\leq C(1+|\X_0|+t)$, and noting that $U(0,t)=0$, it yields that
    \begin{align}
   	  &\int_{st+\X}^{0} (| \X'(t)| |\ubl-\phi^S_{\X}(x,t)|+|H_2(0,t)|)  |U(x,t)|  dx \notag\\
      \leq & C( |\X'| +e^{-c(|\X_0|+t)}+\abs{\p_x^2 U(0,t)} )\int_{st+\X}^0 \int_{x}^0 |\p_x U(y,t)| dy  dx  \notag\\
      \leq & C( |\X'| +e^{-c(|\X_0|+t)}+\abs{\p_x^2 U(0,t)} )(1+|\X_0|+t)^{3/2} \norm{\p_x U} .  \label{eng-basic-52}
   \end{align}
   At the same time, if $st+\X<x\leq \frac{st+\X}{2}$, by separating the integrated interval into $[st+\X,(st+\X)/2]$ and $ [(st+\X)/2,0]$ as the estimates of $H_1$ in \Cref{lem-source}, we have 
   \begin{align*}
   	|H_2(x,t)- H_2(0,t)|
   	&\leq  \int_{x}^{\frac{st+\X}{2}} \abs{h_2(y,t) }dy + \int_{\frac{st+\X}{2}}^0 \abs{h_2(y,t) }dy\\
   	& \leq C\nu e^{\theta_b(st+\X)/2} + C \nu [\phi^S_{\X}(\frac{st+\X}{2},t)-\ubr] \leq C\nu e^{c(st+\X)},
   \end{align*}
   and if $\frac{st+\X}{2}<x\leq 0$, we have
   \begin{align*}
   	|H_2(x,t)- H_2(0,t)| \leq \int_{x}^0 \abs{h_2(y,t) }dy \leq C\nu (\phi^S_{\X}-\ubr) \leq C\nu |\p_x\phi^S_{\X}|. 
   \end{align*}
Therefore, with the Cauchy-Schwartz inequality, one has 
    \begin{align}
   	&\int_{st+\X}^0 |H_2(x,t)-H_2(0,t)| |U(x,t)|  dx \notag\\
   \leq & C \int_{st+\X}^{\frac{st+\X}{2}}\nu  e^{c(st+\X)}  |U| dx 
       + C\int_{\frac{st+\X}{2}}^0 \nu |\p_x\phi^S_{\X}| |U| dx \notag\\
   \leq & C\nu (st+\X)^{\frac{1}{2}}e^{c(st+\X)}\norm{U} 
   + C\nu\left(\int_{\frac{st+\X}{2}}^0 |\p_x\phi^S_{\X}|  dx\right)^{\frac{1}{2}} \|\sqrt{-\p_x\phi^S_{\X}} U \| \notag\\
   \leq & C\nu e^{-c(|\X_0|+t)} \norm{U}  \,.
   	\label{eng-basic-53}
   \end{align}
  In virtue of \eqref{eng-basic-50}, \eqref{eng-basic-51},  \eqref{eng-basic-52}, \eqref{eng-basic-53}, Lemma \ref{lem-shift1}, the H\"{o}lder inequality and Young inequality, one has
  \begin{align}
  	&\abs{\int_{0}^{T}\int_{-\infty}^{0}  w(\phi^S_{\X}) (  -\X' (\ubl-\phi^S_{\X}) -H_2) U  dx dt} \notag\\
  	\leq & \frac{1}{2} \int_{0}^{T} (\|\sqrt{-\p_x\phi^S_{\X}} U \|^2+ \norm{\p_x U}^2 )  d t +\nu e^{-c|\X_0|} \sup_{t\in[0,T]} 
  	 \norm{U(t)}^2 \notag\\
  	&\qquad 
  	+ C \int_{0}^{T} (1+|\X_0|+t)^3 (|\X'|^2+e^{-2c(|\X_0|+t)}+\abs{\p_x^2 U(0,t)}^2)  dt + C \nu e^{-c|\X_0|}\notag\\
  	\leq & \frac{1}{2} \int_{0}^{T} (\|\sqrt{-\p_x\phi^S_{\X}} U \|^2+ \norm{\p_x U}^2 )  d t +\nu e^{-c|\X_0|} \sup_{t\in[0,T]} 
  	\norm{U(t)}^2 \notag\\
  	&\qquad 
  	+ C \Big(e^{-c|\X_0|}+\X_0^{-\beta}\int_{0}^{T} (1+|\X_0|+t)^{\beta+3} \abs{\p_x^2 U(0,t)}^2  dt \Big) \, .
  	\label{eng-basic-5}
  \end{align} 
  
   Therefore, combining \eqref{eng-basic-1'}--\eqref{eng-basic-4} and \eqref{eng-basic-5}, and integrating \eqref{eng-basic-1} in time $[0,t]$, it yields that 
   if $\nu, \varepsilon, |\X_0|^{-1}$ are smaller than some positive constants independent of $T$, then \eqref{eng-basic} is proved.

\end{proof}

\begin{lemma}
	\label{lem-eng-1}
	Under the assumptions of \Cref{prop-apriori-shock}, there exists small positive constant $\nu_3,\varepsilon_3$ independent of $T$ such that, if $\nu<\nu_3$ and $\varepsilon+|\X_0|^{-1}<\varepsilon_3$, then we have
	\begin{equation}\label{eng-1}
		\begin{aligned}
			\sup_{t\in[0,T]}
			&  \abs{U(t)}_{\beta}^2 +\int_{0}^{T}\big(\abs{ U  }_{\beta-1}^2+ \abs{\p_x U}_{\beta}^2 \big)  d t \\
			&\leq C (\norm{U_0}^2 + |\X_0|^{-1}) + C \X_0^{-\underline{\beta}} \int_{0}^{T} (1+|\X_0|+t)^{\beta+3} |\p_x^2 U(0,t)|^2  d t + C\nu,
		\end{aligned} 
	\end{equation}
    where $\underline{\beta}=\min\{\beta,1\}.$
\end{lemma}

\begin{proof}
	Let $\xi=x-st-\X$ and $\xi_*$ satisfying $\phi^S(\xi_*)=(\ubr+\ubl)/2$.
	By multiplying \eqref{eqn-anti-deriv1} with $\xiangle^{\beta} w(\phi^S_{\X}) U$, integrating the resultant on $(-\infty,0)$, using integration by parts, we have
	\begin{align}
		& \dt\int_{-\infty}^{0} \frac{\xiangle^{\beta}}{2}  w(\phi^S_{\X}) U^2 dx 
		+  \int_{-\infty}^{0} \Big[\frac{\xiangle^{\beta-1}}{2} A_\beta U^2 +\xiangle^{\beta}w(\phi^S_{\X})|\p_x U|^2\Big] dx \notag \\
		=	&  - \int_{-\infty}^{0} \frac{\beta}{2} \xiangle^{\beta-2}(\xi-\xi_*) w(\phi^S_\X) (\X'  U^2+2U\p_x U) dx +\int_{-\infty}^{0}  \xiangle^{\beta} w(\phi^S_{\X})  \notag\\
		& \qquad\cdot \{(f'(\phi^S_{\X}) -f'(\usharp)) \p_x U - \X' (\ubl-\phi^S_{\X}) + H_0-H_1-H_2 \} U  dx, \label{eng-1-1}
	\end{align}
	where similar as \eqref{eng-basic-1'} in \Cref{lem-eng-basic},
	\begin{align}
		A_\beta&=\xiangle [(w f_0)''(\phi^S_{\X}) - \X' w'(\phi^S_{\X}) ] (-\p_x \phi^S_{\X})-\beta\frac{\xi-\xi_*}{\xiangle}(wf_0)'(\phi^S_\X)   \notag\\
		&\geq (2-C\varepsilon)\xiangle (-\p_x \phi^S_{\X})-\beta\frac{\xi-\xi_*}{\xiangle}(2\phi^S_\X-\ubr-\ubl)  \notag\\
		&\geq c \quad \text{for small~} \varepsilon . \label{eng-1-1'}
    \end{align}
    Moreover, by \eqref{eqn-assum-shock}, \eqref{def-H0}, \eqref{ineq-H}, the H\"{o}lder inequality and Young inequality, it yields that
    \begin{align}
    	&\abs{\int_{-\infty}^{0} \frac{\beta}{2} \xiangle^{\beta-2}(\xi-\xi_*) w(\phi^S_\X) (\X'  U^2+2U\p_x U) dx} \notag\\
    	&\qquad	\leq \frac{c}{4} \abs{U}_{\beta-1}^2 + C\abs{\p_x U}_{\beta-1}^2
    	\leq \frac{c}{4} \abs{U}_{\beta-1}^2 + \frac{c}{4}\abs{\p_x U}_{\beta}^2 + C \norm{\p_x U}^2, \label{eng-1-2} \\
    	&\Big|\int_{-\infty}^{0} \xiangle^{\beta}w(\phi^S_{\X}) (f'(\phi^S_{\X}) -f'(\usharp))U \p_x U   dx \Big| \notag\\
    	& \qquad \leq C \norm{\xiangle^{\frac{1}{2}} (\ur-\ubr)}_{L^{\infty}} \norm{\sigma_{\X}}_{L^{\infty}} \abs{U}_{\beta-1} \abs{\p_x U}_{\beta} \notag\\
    	&\qquad\leq  C\nu (\abs{U}_{\beta-1}^2+ \abs{\p_x U}_{\beta}^2),      \label{eng-1-3}\\
    	&\abs{\int_{-\infty}^{0} \xiangle^{\beta}w(\phi^S_{\X}) (H_0+H_1) U  dx} \leq  C\varepsilon \abs{\p_x U}_{\beta}^2+C\nu e^{-c(|\X_0|+t)} (1+\abs{U}_{\beta}^2). \label{eng-1-4}
    \end{align}
    It remains to estimate 
    \begin{align}
    	&\abs{\int_{-\infty}^{0}  \xiangle^{\beta} w(\phi^S_{\X})  (-\X' (\ubl-\phi^S_{\X}) -H_2 ) U  dx} \notag\\
    	&\leq C \int_{-\infty}^{0} \xiangle^{\beta}(| \X' | |\ubl-\phi^S_{\X}| +  |H_2|) |U|  dx 
    	=C \left(\int_{-\infty}^{st+\X} +  \int_{st+\X}^0\right) \,. \label{eng-1-50}
    \end{align}
   (i) For $x\leq st+\X$, noting that 
   $|\ubl-\phi^S(\xi)|\leq C|\p_x \phi^S(\xi)|$ for $\xi \leq 0$ and $\abs{H_2(x,t)} \leq C  \nu e^{-c(|\X_0|+t)} |\p_x \phi^S_{\X}| $, it holds that
   and thus 
   \begin{align}
   	&\int_{-\infty}^{st+\X}  \xiangle^{\beta}(| \X' | |\ubl-\phi^S_{\X}| +  |H_2|) |U| dx \notag \\
   	&\leq  \frac{c}{12} \abs{U}_{\beta-1}^2 +
   	C (| \X' |+\nu e^{-c(|\X_0|+t)})^2\int_{-\infty}^{st+\X} \xiangle^{\beta+1} |\p_x \phi^S_{\X}|^2  dx \notag\\
   	& \leq \frac{c}{12} \abs{U}_{\beta-1}^2 +
   	C (| \X' |+\nu e^{-c(|\X_0|+t)})^2.\label{eng-1-51}
   \end{align}
   (ii) For $st+\X<x\leq 0$, reviewing that $H_2(0,t)= \mathcal{O} ( |\X'| +e^{-c(|\X_0|+t)}+\abs{\p_x^2 U(0,t)}  )$ (obtained in \Cref{lem-eng-basic}), 
   it yields that
   \begin{align}
   	&\int_{st+\X}^{0} \xiangle^{\beta} (| \X'(t)| |\ubl-\phi^S_{\X}(x,t)|+|H_2(0,t)|)  |U(x,t)|  dx \notag\\
   	\leq & C( |\X'| +e^{-c(|\X_0|+t)}+\abs{\p_x^2 U(0,t)} )  \big(\int_{st+\X}^0 \xiangle^{\beta+1}  dx\big)^{1/2} \abs{U}_{\beta-1} \notag\\
   	\leq & \frac{c}{12}\abs{U}_{\beta-1}^2 +C( |\X'| +e^{-c(|\X_0|+t)}+\abs{\p_x^2 U(0,t)} )^2 (1+|\X_0|+t)^{\beta+2}.  \label{eng-1-52}
   \end{align}
   At the same time, reviewing that $|H_2(x,t)- H_2(0,t)|\leq C\nu e^{c(st+\X)}$ if $st+\X<x\leq \frac{st+\X}{2}$ and $|H_2(x,t)- H_2(0,t)|\leq C\nu |\p_x\phi^S_{\X}|$ if $\frac{st+\X}{2}<x\leq 0$ (obtained in \Cref{lem-eng-basic}), with the Cauchy-Schwartz inequality, one has 
   \begin{align}
   	&\int_{st+\X}^0 \xiangle^{\beta} |H_2(x,t)-H_2(0,t)| |U(x,t)|  dx \notag\\
   	\leq & C \int_{st+\X}^{\frac{st+\X}{2}} \xiangle^{\beta} \nu   e^{c(st+\X)}  |U| dx 
   	+ C\int_{\frac{st+\X}{2}}^0 \xiangle^{\beta} \nu |\p_x\phi^S_{\X}| |U| dx \notag\\
   	\leq & C\nu (st+\X)^{\frac{\beta+2}{2}}e^{c(st+\X)}\abs{U}_{\beta-1} 
   	+ C\nu\left(\int_{\frac{st+\X}{2}}^0 \xiangle^{\beta+1} |\p_x\phi^S_{\X}|^2  dx\right)^{\frac{1}{2}} \abs{U}_{\beta-1}  \notag\\
   	\leq & C\nu (st+\X)^{\frac{\beta+2}{2}}e^{c(st+\X)}\abs{U}_{\beta-1} \leq  C\nu e^{-c(|\X_0|+t)} \abs{U}_{\beta-1}  \,.
   	\label{eng-1-53}
   \end{align}
   In virtue of \eqref{eng-1-50}, \eqref{eng-1-51},  \eqref{eng-1-52}, \eqref{eng-1-53}, Lemma \ref{lem-shift1}, the H\"{o}lder inequality and Young inequality, one has
   \begin{align}
   	&\abs{\int_{0}^{T}\int_{-\infty}^{0} \xiangle^{\beta} w(\phi^S_{\X}) (  -\X' (\ubl-\phi^S_{\X}) -H_2) U  dx dt} \notag\\
   	\leq & \frac{c}{4} \int_{0}^{T} \abs{U}_{\beta-1}^2  d t 
   	+ C \int_{0}^{T} (1+|\X_0|+t)^{\beta+2} (|\X'|^2+e^{-2c(|\X_0|+t)}+\abs{\p_x^2 U(0,t)}^2)  dt \notag\\
   	&\qquad+ C \nu e^{-c|\X_0|}\notag\\
   	\leq & \frac{c}{4} \int_{0}^{T} \abs{U}_{\beta-1}^2  d t  
   	+ C e^{-c|\X_0|}+C\X_0^{-1}\int_{0}^{T} (1+|\X_0|+t)^{\beta+3} \abs{\p_x^2 U(0,t)}^2  dt \, .
   	\label{eng-1-5}
   \end{align} 

   Therefore, combining \eqref{eng-1-1'}--\eqref{eng-1-4}, \eqref{eng-1-5} and \Cref{lem-eng-basic}, and integrating \eqref{eng-1-1} in time $[0,t]$, it yields that 
   if $\nu, \varepsilon, |\X_0|^{-1}$ are smaller than some positive constants independent of $T$, then \eqref{eng-1} is proved.
\end{proof}

\subsubsection{Higher order estimates}

When considering the spatially derivatives of \eqref{eqn-anti-deriv1}, the derivatives of $\X'(\ubl-\phi^S_{\X})-H_2$ are easier to 

\begin{lemma}
	\label{lem-eng-2}
	Under the assumptions of \Cref{prop-apriori-shock}, there exists small positive constant $\nu_4,\varepsilon_4$ independent of $T$ such that, if $\nu<\nu_4$ and $\varepsilon+|\X_0|^{-1}<\varepsilon_4$, then we have
	\begin{equation}\label{eng-2}
		\begin{aligned}
			\sup_{t\in[0,T]} & \abs{\p_x U(t)}_{\beta+1}^2 +\int_{0}^{T}  \abs{\p_x^2 U}_{\beta+1}^2  d t \\
			&\leq C (\norm{U_0}_{H_{\beta}^1}^2 + |\X_0|^{-1}) + C \X_0^{-\underline{\beta}} \int_{0}^{T}(1+|\X_0|+t)^{\beta+3} |\p_x^2 U(0,t )|^2  d t + C\nu.
		\end{aligned} 
	\end{equation}
\end{lemma}

\begin{proof}
Differentiate \eqref{eqn-anti-deriv1} with respect to $x$ gives 
\begin{equation}
	\label{eqn-anti-deriv2}
		\p_t\p_x U+\p_x(f'(\usharp)\p_x U)-\p_x^3 U=\X'\p_x\phi^S_{\X} + \p_x H_0-\p_x H_1-h_2\,, 
\end{equation}
By multiplying \eqref{eqn-anti-deriv2} with $\xiangle^{\beta+1}\p_x U$, integrating the resultant on $(-\infty,0)$, using integration by parts, we have
\begin{align}
	& \dt\int_{-\infty}^{0} \frac{\xiangle^{\beta+1}}{2}  |\p_x U|^2 dx 
     +  \int_{-\infty}^{0} \xiangle^{\beta+1}|\p_x^2 U|^2 dx \notag\\
   =&-  \Big\{\xiangle^{\beta+1}\big(\frac{1}{2}f'(\usharp)|\p_x U|^2-\p_x U\p_x^2 U\big)\Big\}\Big\vert_{x=0} \notag \\	
    &+\int_{-\infty}^{0} \frac{\beta+1}{2} \xiangle^{\beta-1}(\xi-\xi_*) 
      [(f'(\usharp)-s-\X')  |\p_x U|^2-2\p_x U\p_x^2 U] dx  \notag	\\
    &- \int_{-\infty}^{0} \frac{1}{2} \xiangle^{\beta+1} f''(\usharp)\p_x \usharp |\p_x U|^2 dx \notag\\
    &  + \int_{-\infty}^{0}  \xiangle^{\beta+1}  ( \X' \p_x\phi^S_{\X} +\p_x H_0-\p_x H_1-h_2 ) \p_x U  dx . \label{eng-2-1}
\end{align}
where $\p_x H_0 = \mathcal{O} (|\p_x \usharp||\p_x U|^2+|\p_x U\p_x^2 U|)$.
Thanks to \eqref{bc0} and \eqref{def-ansatz-shock}, we have 
\begin{equation*}
	\label{ic-1}
	\p_x U(0,t)=u_b(t)-\usharp(0,t)=(u_b(t)-\ubl)(1-\sigma(-st-\X))=\mathcal{O}(e^{\theta_s(st+\X)}),
\end{equation*}
and thus it follows that
\begin{align}
	&\abs{ \Big\{\xiangle^{\beta+1}\big(\frac{1}{2}f'(\usharp)|\p_x U|^2-\p_x U\p_x^2 U\big) \Big\}	\Big\vert_{x=0} } \notag \\
	&\qquad\leq C(st+\X)^{\beta+1} (e^{2c(st+\X)} + |\p_x^2 U(0,t)|^2) \notag\\
	&\qquad\leq C(1+|\X_0|+t)^{\beta+1} (e^{-c(|\X_0|+t)} + |\p_x^2 U(0,t)|^2)\,. \label{eng-2-2}
\end{align}
In virtue of \eqref{eqn-assum-shock}, \eqref{def-H0}, \eqref{ineq-H}, the H\"{o}lder inequality and Young inequality, it yields that
\begin{align}
	&\abs{\int_{-\infty}^{0}  \xiangle^{\beta-1}(\xi-\xi_*) 
		[(f'(\usharp)-s-\X')  |\p_x U|^2-2\p_x U\p_x^2 U] dx  } \notag\\
	&\qquad \leq \frac{1}{2}\abs{\p_x^2 U}_{\beta+1}^2+ C \abs{\p_x U}_{\beta}^2, \label{eng-2-3}\\
	&\abs{\int_{-\infty}^{0}  \xiangle^{\beta+1} f''(\usharp)\p_x \usharp |\p_x U|^2 dx } \leq \norm{\xiangle \p_x \usharp}_{L^\infty} \abs{\p_x U}_{\beta}^2  \leq  C\abs{\p_x U}_{\beta}^2,  \label{eng-2-4} \\
	&\abs{\int_{-\infty}^{0}  \xiangle^{\beta+1}    \p_x H_0  \p_x U  dx } \notag\\
	&\qquad\leq C \norm{\xiangle^{\frac{1}{2}} \p_x U}_{L^{\infty}} (\norm{\xiangle^{\frac{1}{2}} \p_x \usharp}_{L^{\infty}}\abs{\p_x U}_{\beta}^2 +\abs{\p_x U}_{\beta}\abs{\p_x^2 U}_{\beta+1})\notag\\
	&\qquad\leq C\varepsilon (\abs{\p_x^2 U}_{\beta+1}^2+\abs{\p_x U}_{\beta}^2), \label{eng-2-5}  \\
	&\abs{\int_{0}^{T}\int_{-\infty}^{0}  \xiangle^{\beta+1}  ( \X' \p_x\phi^S_{\X}-\p_x H_1-h_2 ) \p_x U  dx dt} \notag \\
	&\qquad \leq C\int_{0}^{T}(\abs{\X'\p_x\phi^S_{\X} }_{\beta+1} +  \abs{\p_x H_1 }_{\beta+1} + \abs{h_2 }_{\beta+1} )\abs{\p_x U}_{\beta+1}  dt \notag\\
	&\qquad \leq C\int_{0}^{T}(\abs{\X' } +  \nu e^{-c(|\X_0|+t)})\abs{\p_x U}_{\beta+1}  dt \notag\\
	&\qquad \leq \frac{1}{2}\sup_{t\in[0,T]}\abs{\p_x U}_{\beta+1}^2 + C e^{-c|\X_0|}+C\X_0^{-1}\int_{0}^{T} (1+|\X_0|+t)^{\beta+3} \abs{\p_x^2 U(0,t)}^2  dt . \label{eng-2-6}
\end{align} 

Therefore, combining \eqref{eng-2-2}--\eqref{eng-2-6}, and \Cref{lem-shift1,lem-eng-basic,lem-eng-1}, and integrating \eqref{eng-2-1} in time $[0,t]$, it yields that 
if $\nu, \varepsilon, |\X_0|^{-1}$ are smaller than some positive constants independent of $T$, then \eqref{eng-2} is proved.
\end{proof}

\begin{lemma}
	\label{lem-eng-3}
	Under the assumptions of \Cref{prop-apriori-shock}, there exists small positive constant $\nu_5,\varepsilon_5$ independent of $T$ such that, if $\nu<\nu_5$ and $\varepsilon+|\X_0|^{-1}<\varepsilon_5$, then we have
	\begin{equation}\label{eng-3}
		\begin{aligned}
			\sup_{t\in[0,T]}
			&  \abs{\p_x^2 U(t)}_{\beta+2}^2 +\int_{0}^{T}\big\{ (1+|\X_0|+t)^{\beta+2} |\p_x^2 U(0,t )|^2   + \abs{\p_x^3 U}_{\beta+2}^2 \big\}  d t \\
			&\leq C (\norm{U_0}_{H_{\beta}^2}^2 + |\X_0|^{-1})  + C \X_0^{-\underline{\beta}} \int_{0}^{T} (1+|\X_0|+t)^{\beta+3} |\p_x^2 U(0,t )|^2  d t + C\nu.
		\end{aligned} 
	\end{equation}
\end{lemma}

\begin{proof}
	Differentiate \eqref{eqn-anti-deriv2} with respect to $x$ gives 
	\begin{equation}
		\label{eqn-anti-deriv3}
		\p_t\p_x^2 U+\p_x(f'(\usharp)\p_x^2 U+f''(\usharp)\p_x \usharp \p_x U)-\p_x^4 U=\X'\p_x^2 \phi^S_{\X} + \p_x^2 H_0-\p_x^2 H_1-\p_x h_2\,, 
	\end{equation}
	By multiplying \eqref{eqn-anti-deriv2} with $\xiangle^{\beta+2}\p_x^2 U$, integrating the resultant on $(-\infty,0)$, using integration by parts, we have
	\begin{align}
		& \dt\int_{-\infty}^{0} \frac{\xiangle^{\beta+2}}{2}  |\p_x^2 U|^2 dx 
		+  \int_{-\infty}^{0} \xiangle^{\beta+2}|\p_x^3 U|^2 dx \notag\\
	    & + \Big\{\xiangle^{\beta+2}\big(\frac{1}{2}f'(\usharp)|\p_x^2 U|^2-\p_x^2 U\p_x^3 U\big)\Big\}\Big\vert_{x=0} \notag \\	
		=&\int_{-\infty}^{0} \frac{\beta+2}{2} \xiangle^{\beta}(\xi-\xi_*) 
		[(f'(\usharp)-s-\X')  |\p_x^2 U|^2-2\p_x^2 U\p_x^3 U] dx  \notag	\\
		&- \int_{-\infty}^{0}  \xiangle^{\beta+2} \big[\frac{3}{2}f''(\usharp)\p_x\usharp \p_x^2 U + f''(\usharp) \p_x^2 \usharp \p_x U + f'''(\usharp)(\p_x \usharp)^2\p_x U\big]\p_x^2 U dx \notag\\
		&  + \int_{-\infty}^{0}  \xiangle^{\beta+2}  ( \X' \p_x^2\phi^S_{\X} +\p_x^2 H_0-\p_x^2 H_1-\p_x h_2 ) \p_x^2 U  dx , \label{eng-3-1}
	\end{align}
    where $\p_x^2 H_0 = \mathcal{O} ((|\p_x^2 \usharp|+|\p_x \usharp|)|\p_x U|^2+|\p_x^2 U|^2+|\p_x U\p_x^3 U|)$.
	It follows from \eqref{eqn-0}, \eqref{bc0} and \eqref{def-ansatz-shock} that for small $\nu$,
	\begin{align}
		\p_x^2 u(0,t) &=  u_b'(t) + f'(u_b(t)) \p_x u (0,t), \notag\\
		\p_x^3 U(0,t) 
		&= \big(\p_x^2 u-\p_x^2 \ur+\p_x^2 \ur(1-\sigma_{\X})-2\p_x \ur\p_x \sigma_{\X}-(\ur-\ubl)\p_x^2 \sigma_{\X} \big)\Big\vert_{x=0}  \notag\\
		&=f'(u_b(t))\p_x^2 U(0,t)+ \mathcal{O}(e^{\theta_s(st+\X)}), \label{ic-3}
	\end{align}
	and thus
	\begin{align}
		&  \Big\{\xiangle^{\beta+2}\big(\frac{1}{2}f'(\usharp)|\p_x^2 U|^2-\p_x^2 U\p_x^3 U\big) \Big\}	\Big\vert_{x=0} \notag \\
		&\qquad\geq (st+\X)^{\beta+2} \Big\{\big[\frac{1}{2}f'(\usharp(0,t))-f'(u_b(t)) \big] |\p_x^2 U(0,t)|^2 -C (e^{2c(st+\X)})\Big\} \notag\\
		&\qquad\geq c (1+|\X_0|+t)^{\beta+2}  |\p_x^2 U(0,t)|^2 -C (e^{-c(|\X_0|+t)}) \,. \label{eng-3-2}
	\end{align}
	In virtue of \eqref{eqn-assum-shock}, \eqref{def-H0}, \eqref{ineq-H}, the H\"{o}lder inequality and Young inequality, it yields that
	\begin{align}
		&\abs{\int_{-\infty}^{0}  \xiangle^{\beta}(\xi-\xi_*) 
			[(f'(\usharp)-s-\X')  |\p_x^2 U|^2-2\p_x^2 U\p_x^3 U] dx  } \notag\\
		&\qquad \leq \frac{1}{2}\abs{\p_x^3 U}_{\beta+2}+ C \abs{\p_x^2 U}_{\beta+1}, \label{eng-3-3}\\
		&\abs{ \int_{-\infty}^{0}  \xiangle^{\beta+2} \big[\frac{3}{2}f''(\usharp)\p_x\usharp \p_x^2 U + f''(\usharp) \p_x^2 \usharp \p_x U + f'''(\usharp)(\p_x \usharp)^2\p_x U\big]\p_x^2 U dx } \notag \\
		&\qquad \leq \norm{\xiangle^{3/2} \p_x \usharp}_{L^\infty} (\abs{\p_x^2 U}_{\beta+1}^2+\abs{\p_x U}_{\beta}^2)
		\leq  C(\abs{\p_x^2 U}_{\beta+1}^2+\abs{\p_x U}_{\beta}^2),  \label{eng-3-4} \\
		&\abs{\int_{-\infty}^{0}  \xiangle^{\beta+2}    \p_x^2 H_0  \p_x^2 U  dx } \notag\\
		&\qquad \leq C \norm{\xiangle \p_x^2 U}_{L^{\infty}} (\norm{\xiangle (|\p_x^2 \usharp|+|\p_x \usharp|)}_{L^{\infty}}\abs{\p_x U}_{\beta}^2 +\abs{\p_x^2 U}_{\beta+1}^2 \notag\\
		&\qquad\quad +\abs{\p_x U}_{\beta}\abs{\p_x^3 U}_{\beta+2}) 
		\leq C\varepsilon (\abs{\p_x^3 U}_{\beta+2}^2+\abs{\p_x^2 U}_{\beta+1}^2 + \abs{\p_x U}_{\beta}^2 ), \label{eng-3-5}  \\
		&\abs{\int_{0}^{T}\int_{-\infty}^{0}  \xiangle^{\beta+2}  ( \X' \p_x^2\phi^S_{\X}-\p_x^2 H_1-\p_x h_2 ) \p_x^2 U  dx dt} \notag \\
		&\qquad \leq C\int_{0}^{T}(\abs{\X'\p_x^2\phi^S_{\X} }_{\beta+2} +  \abs{\p_x^2 H_1 }_{\beta+2} + \abs{\p_x h_2 }_{\beta+2} )\abs{\p_x^2 U}_{\beta+2} dt \notag\\
		&\qquad \leq \frac{1}{2}\sup_{t\in[0,T]}\abs{\p_x^2 U}_{\beta+2}^2 + C e^{-c|\X_0|}+C\X_0^{-1}\int_{0}^{T} (1+|\X_0|+t)^{\beta+3} \abs{\p_x^2 U(0,t)}^2  dt .
		 \label{eng-3-6}
	\end{align} 
	
	Therefore, combining \eqref{eng-3-2}--\eqref{eng-3-6}, and \Cref{lem-shift1,lem-source,lem-eng-basic,lem-eng-1,lem-eng-2}, and integrating \eqref{eng-3-1} in time $[0,t]$, it yields that 
	if $\nu, \varepsilon, |\X_0|^{-1}$ are smaller than some positive constants independent of $T$, then \eqref{eng-3} is proved.
    
\end{proof}

\begin{lemma}
	\label{lem-eng-4}
	Under the assumptions of \Cref{prop-apriori-shock}, there exists small positive constant $\nu_6,\varepsilon_6$ independent of $T$ such that, if $\nu<\nu_6$ and $\varepsilon+|\X_0|^{-1}<\varepsilon_6$, then we have
	\begin{equation}\label{eng-4}
		\begin{aligned}
			\sup_{t\in[0,T]}
			& \{ \abs{\p_x^3 U(t)}_{\beta+3}^2 +  (1+|\X_0|+t)^{\beta+3} |\p_x^2 U(0,t )|^2\}\\
			& + \int_{0}^{T}\big\{ (1+|\X_0|+t)^{\beta+3} |\p_x^2 U(0,t)|^2   + \abs{\p_x^4 U}_{\beta+3}^2 \big\}  d t \\
			& \leq C (\norm{U_0}_{H_{\beta}^3}^2 + |\X_0|^{-1}) + C\nu.
		\end{aligned} 
	\end{equation}
\end{lemma}

\begin{proof}
	Differentiate \eqref{eqn-anti-deriv3} with respect to $x$ gives 
	\begin{equation}
		\label{eqn-anti-deriv4}
		\begin{aligned}
			\p_t\p_x^3 U+\p_x(f'(\usharp)\p_x^3 U)-\p_x^5 U=&-\p_x[f''(\usharp)\p_x \usharp \p_x^2 U +\p_x(f''(\usharp)\p_x \usharp \p_x U)]\\
			&+\X'\p_x^3 \phi^S_{\X} + \p_x^3 H_0-\p_x^3 H_1-\p_x^2 h_2   \,, 
		\end{aligned}
	\end{equation}
	By multiplying \eqref{eqn-anti-deriv4} with $\xiangle^{\beta+3}\p_x^3 U$, integrating the resultant on $(-\infty,0)$, using integration by parts, we have
	\begin{align}
		& \dt\int_{-\infty}^{0} \frac{\xiangle^{\beta+3}}{2}  |\p_x^3 U|^2 dx 
		+  \int_{-\infty}^{0} \xiangle^{\beta+3}|\p_x^4 U|^2 dx \notag\\
		& + \Big\{\xiangle^{\beta+3}\big(\frac{1}{2}f'(\usharp)|\p_x^3 U|^2-\p_x^3 U\p_x^4 U\big)\Big\}\Big\vert_{x=0} \notag \\	
		=&\int_{-\infty}^{0} \frac{\beta+3}{2} \xiangle^{\beta+1}(\xi-\xi_*) 
		[(f'(\usharp)-s-\X')  |\p_x^3 U|^2-2\p_x^3 U\p_x^4 U] dx  \notag	\\
		&- \int_{-\infty}^{0}  \xiangle^{\beta+3}\p_x^3 U   [\frac{1}{2}f''(\usharp)\p_x \usharp \p_x^3 U +\text{right had terms of~}\eqref{eqn-anti-deriv4}]. \label{eng-4-1}
	\end{align} 
	It follows from \eqref{eqn-0}, \eqref{bc0} and \eqref{def-ansatz-shock} that 
	\begin{align}
		\p_x^3 u(0,t) &=  \p_t\p_x u(0,t) + f'(u_b) \p_x^2 u (0,t)+f''(u_b) (\p_x^2 u (0,t))^2, \notag\\
		\p_x^4 U(0,t) 
		&= \big\{\p_x^3 u-\p_x^3 \ur+\mathcal{O}(1) (|1-\sigma_{\X}|+|\p_x\sigma_{\X}|+|\p_x^2\sigma_{\X}|+|\p_x^3\sigma_{\X}|)\big\}\Big\vert_{x=0}  \notag\\
		&=\p_t\p_x^2 U(0,t) + f'(u_b)\p_x^3 U(0,t) + f''(u_b) |\p_x^2 U(0,t)|^2 \notag\\
		&\qquad+ 2f''(u_b) \p_x^2 U(0,t)\p_x^2 \usharp(0,t) + \mathcal{O}(e^{\theta_s(st+\X)}) \notag\\
		&=\p_t\p_x^2 U(0,t) + f'(u_b)\p_x^3 U(0,t)+ \mathcal{O}(1) ( |\p_x^2 U(0,t)|^2+  e^{\theta_s(st+\X)}), \label{ic-4}
	\end{align}
	and thus with $\p_x^3 U(0,t)=f'(u_b(t))\p_x^2 U(0,t)+ \mathcal{O}(e^{c(st+\X)})$ in \eqref{ic-3} and the Young inequality, we have 
	\begin{align}
		& \Big\{\xiangle^{\beta+3}\big(\frac{1}{2}f'(\usharp)|\p_x^3 U|^2-\p_x^3 U\p_x^4 U\big) \Big\}	\Big\vert_{x=0} \notag \\
	=	& (st+\X)^{\beta+3} \Big\{ -\p_x^3 U(0,t)\p_t\p_x^2 U(0,t) -\frac{1}{2}f'(u_b)  |\p_x^3 U(0,t)|^2 \notag \\
	    &\quad - \mathcal{O}(1)|\p_x^3 U(0,t)| ( |\p_x^2 U(0,t)|^2+  e^{\theta_s(st+\X)})\Big\} \notag\\
	=	& (st+\X)^{\beta+3} \Big\{ -\frac{1}{2} \Big[ \dt (f'(u_b)|\p_x^2 U(0,t)|^2) + (f'(u_b))^3|\p_x^2 U(0,t)|^2\Big] \notag\\
		&\quad  + \dt (\mathcal{O}(1) e^{c(st+\X)} \p_x^2 U(0,t) ) - \mathcal{O}(1) ((\varepsilon+\nu)|\p_x^2 U(0,t)|^2+ e^{2c(st+\X)})\Big\}\,. \label{eng-4-2}
	\end{align}

	The right hand terms in \eqref{eng-4-1} are lower order terms, and can be estimates as \Cref{lem-eng-2,lem-eng-3}. Therefore, combining \eqref{eng-4-2}, \Cref{lem-shift1,lem-source,lem-eng-basic,lem-eng-1,lem-eng-2,lem-eng-3}, and integrating \eqref{eng-4-1} in time $[0,t]$, it yields that 
	\begin{align*}
		\sup_{t\in[0,T]}
		& \{ \abs{\p_x^3 U(t)}_{\beta+3}^2 +  (1+|\X_0|+t)^{\beta+3} |\p_x^2 U(0,t )|^2\}\\
		& + \int_{0}^{T}\big\{ (1+|\X_0|+t)^{\beta+3} |\p_x^2 U(0,t)|^2   + \abs{\p_x^4 U}_{\beta+3}^2 \big\}  d t \\
		 \leq& C (\norm{U_0}_{H_{\beta}^3}^2 + |\X_0|^{-1}) +C (\X_0^{-\underline{\beta}}+\varepsilon+\nu) \int_{0}^{T} (1+|\X_0|+t)^{\beta+3} |\p_x^2 U(0,t )|^2  d t + C\nu.
	\end{align*} 
     It follows that if $\nu, \varepsilon, |\X_0|^{-1}$ are smaller than some positive constants independent of $T$, then \eqref{eng-4} is proved.
	
\end{proof}

\begin{proof}[Proof of Theorem \ref{thm-shock-viscous}]
	Combining all \Cref{lem-eng-basic,lem-eng-1,lem-eng-2,lem-eng-3,lem-eng-4} with a continuity argument, one can obtain \eqref{ineq-apriori-shock} and the rigidity of the a priori assumption \eqref{eqn-assum-shock}, i.e. \Cref{prop-apriori-shock} is proved.
	
	Hence, by the standard local-existence theory, the problem \eqref{eqn-anti-deriv} and \eqref{eqn-shift1} admits a unique global solution $ (U,\X) \in \mathcal{S}(0,\infty) \times \mathcal{Z}(0,\infty)$ for small $\norm{U_0}_{H_{\beta}^3}$, $|\X_0|^{-1}$ and $\nu$, and the estimate \eqref{ineq-apriori-shock} holds for all $T\in [0,\infty)$. So Theorem \ref{thm-anti-deriv} can be  completed.
	
	Furthermore, the estimate \eqref{ineq-apriori-shock} and the equation \eqref{eqn-anti-deriv} imply that
    $$\int_{0}^{\infty} \abs{\p_x U(t)}_{\beta}^2 dt + \int_{0}^{\infty} \abs{\frac{d}{dt} \abs{\p_x U(t)}_{\beta}^2} dt <\infty,
    $$	
    which together with \eqref{ineq-apriori-shock} shows that 
     $$\norm{U}_{L^\infty}+\norm{\p_x U}_{L^\infty}\rightarrow 0 \quad \text{as~} t\rightarrow\infty.$$
	Note that by separating the $\R_-$ into $(-\infty,(st+\X)/2]$ and $ [(st+\X)/2,0]$ as above estimates, it holds that
	\begin{equation}
		\label{ineq-two ansatz}
			\norm{\usharp - (\ur - \ubr + \phi^S_{\X}) }_{L^\infty}
			= \norm{(\ur - \ubr)(1-\sigma_{\X})  }_{L^\infty}
			\leq  C e^{-c(|\X_0|+t)} .	
	\end{equation}
	Then \eqref{asymptotic behavior} holds and hence Theorem \ref{thm-shock-viscous} is proved.
	
\end{proof}

\section*{}
\noindent\textbf{Acknowledgements}. The author would like to thank Professor Zhouping Xin for helpful discussions, and to thank the Institute of Mathematical Sciences, The Chinese University of Hong Kong for its hospitality during the visit, in which the research of this work was partially studied.
The research is supported by the National Natural Science Foundation of China (No. 12471210).

\vspace{0.3cm}

\noindent\textbf{Claims}. 
The author claims that this paper has no associated data, and that there is no conflict of interest.

\vspace{1cm}

\bibliographystyle{plain}
\bibliography{stability}

\end{document}